\documentclass[12pt]{amsart}
\usepackage{amssymb}
\usepackage{tikz}
\usepackage{hyperref}
\usepackage{mathtools}
\DeclarePairedDelimiter{\floor}{\lfloor}{\rfloor}
\allowdisplaybreaks

\def\II{\mathrm{II}}
\def\Atilde{\widetilde{A}}

\def\Aut{\mathop{\rm Aut}\nolimits}

\def\twobyone#1#2{\bigl(\begin{smallmatrix}#1\\#2\end{smallmatrix}\bigr)
    }
\def\twobytwo#1#2#3#4{\bigl(\begin{smallmatrix}#1&#2\\#3&#4\end{smallmatrix}\bigr)
    }
\def\gend#1{\langle#1\rangle}

\def\calV{\mathcal{V}}

\def\calR{\mathcal{R}}

\def\calD{\mathcal{D}}
\def\calN{\mathcal{N}}

\def\R{\mathbb{R}}
\def\O{\mathrm{O}}
\def\SO{\mathrm{SO}}
\def\SOup{\SO^{\uparrow}}
\def\Oup{\O^{\uparrow}}
\def\calO{\mathcal{O}}
\def\Z{\mathbb{Z}}
\def\Q{\mathbb{Q}}
\def\tensor{\otimes}
\def\sset{\subseteq}
\def\set#1#2{\{#1\mathbin|#2\}}
\def\Hom{\mathop{\rm Hom}\nolimits}

\def\D{\Delta}
\def\SOup{\mathrm{SO}^{\uparrow}}
\let\iso\cong
\def\cong{\equiv}

\newtheorem{theorem}{Theorem}
\numberwithin{theorem}{section}
\numberwithin{equation}{section}
\numberwithin{figure}{section}
\newtheorem{lemma}[theorem]{Lemma}
\newtheorem{defnlemma}[theorem]{Definition/Lemma}
\newtheorem{algorithm}[theorem]{Algorithm}

\newtheorem{problem}[theorem]{Problem}

\newtheorem*{SolutionToSphericalIterative}{Solution}
\newtheorem*{SolutionToCuspidalBatch0}{Solution to Problem~\ref{ProbCuspidalBatch0}}
\newtheorem*{SolutionToEdgewalking}{Solution to Problem~\ref{ProbEdgewalking}}

\theoremstyle{remark}
\newtheorem{remark}[theorem]{Remark}

\begin{document}

\title[]{An alternative to Vinberg's algorithm}
\author{Daniel Allcock}
\thanks{Supported by Simons Foundation Collaboration Grant 429818.}
\address{Department of Mathematics\\University of Texas, Austin}
\email{allcock@math.utexas.edu}
\urladdr{http://www.math.utexas.edu/\textasciitilde allcock}
\date{August 31, 2021}

\begin{abstract}
    Vinberg's algorithm is the main method for finding
    fundamental domains for reflection groups acting
    on hyperbolic space.  Experience shows that it can
    be slow.  We explain why this should be expected,  and
    prove this slowness in some cases.  And we provide an
    alternative algorithm that should be much faster.
    It depends on an algorithm for finding vectors
    with small positive norm 
    in indefinite binary quadratic forms, of independent
    interest.
\end{abstract}

\maketitle

\section{Introduction}
\label{sec-introduction}

\noindent
The Weyl group $W(L)$ of a Lorentzian lattice $L$ 
is generated by reflections and acts
properly discontinuously on hyperbolic space.
Any one of its chambers serves as a fundamental domain, and 
computing $W(L)$ amounts to explicitly describing a chamber.
Since the early 1970's, Vinberg's algorithm has been the 
tool for doing this.  In this paper we present a new
method, which has several advantages.
In very general terms, Vinberg's algorithm is like
solving Pell's equation
$x^2-ny^2=1$ by brute force, first trying $y=1$, then $y=2$, etc.
 Ours is like the exponentially faster
solution using continued fractions.
We hope to apply it to the problem of classifying all
Lorentzian lattices which are reflective (ie have
chambers of finite hyperbolic volume).

Vinberg's algorithm requires a choice of a point $k$ of hyperbolic
space,
a choice of chamber $C$ containing~$k$,
and a way to enumerate all reflections in $W(L)$
whose mirrors lie at any chosen hyperbolic distance from~$k$.  The output
is the list of all simple roots of~$C$.  If there are only
finitely many then the algorithm recognizes this and terminates.
If not, then
the algorithm might not terminate, but will still find all
the simple roots
if left to run forever.

Our algorithm requires $k$ to be
a vertex,
and only finds the simple roots corresponding to
$k$'s component of the topological boundary $\partial C$.  
Neither of these limitations is serious.
First, 
$k$ has always been chosen to be a vertex of~$C$,
except in a few special situations---the main example being Conway's group
$W(\II_{25,1})$, whose chamber has infinite volume and infinitely
many facets \cite{Conway}.  Second,
$\partial C$ is connected unless $C$
has infinite volume.   The usual method of recognizing 
that $C$ has infinite volume
is to find enough vertices of $C$ that one
can construct  infinitely many
diagram automorphisms of~$C$.    This 
applies without change in our situation
(see remark~\ref{RkInfiniteVolumeTermination}). 
On the other hand, our method
only works for lattices over~$\Z$.
Also, we do use
Vinberg's algorithm in a limited way, for working out the
simple roots
at a corner of the chamber, given the corner.

When comparing algorithms, one usually analyzes 
the asymptotics of the time or space
required, as a function of the size of the input.  
This does not make sense in this case, because
there is no
upper bound on the size of the answer,
in terms of the size of the input.
Furthermore, Nikulin showed
that there are only
finitely many reflective lattices \cite{NikulinFinitelyMany}\cite{Esselmann}.  So there 
are constant upper bounds on the run-times of both algorithms
(ours and Vinberg's),
when they are applied to reflective lattices.  
And applied to
non-reflective lattices, both algorithms run forever. 
So there is no asymptotic way to compare them.  (In a few
nonreflective
examples like $\II_{25,1}$, Vinberg's algorithm ``terminates''
after finding infinitely many simple roots in finitely many batches \cite{Conway}\cite{BorcherdsLikeTheLeechLattice}.  
There are only finitely many examples like this \cite{NikulinFinitelyManyParabolic}.)

So our main goal is to prove the correctness of our
algorithm.  See remark~\ref{RkPerformance} for a heuristic 
argument that  it is much faster than
Vinberg's. 

Experience shows that for some lattices,
Vinberg's algorithm can take a very long
time.  It also slows down dramatically as
the dimension of~$L$ increases. 
In section~\ref{SecSlowness}
we prove that Vinberg's algorithm can require astronomical
time for fairly simple lattices.  Our
algorithm was motivated by the observed slowness,
this proven slowness,
and our pessimism about the performance of existing
implementations of Vinberg's algorithm in large dimensions ($\dim(L)>6$ or so).

Our basic idea is to ``walk along the edges of~$C$''.  That is,
given a vertex, 
we 
find the vertices at the other ends of the edges emanating from it.
This amounts to searching for small-but-positive-norm vectors
in $2$-dimensional Lorentzian lattices, which we do using
a variation of the standard method for solving Pell's equation 
$x^2-ny^2=1$.  This appears in
section~\ref{SecShortVectors}; we hope it is of independent interest.
After finding these neighboring 
vertices, we find their neighbors, and then their neighbors' neighbors, etc.
We stop when the endpoints of all edges are known
(or all the walking tires us out).  
The stopping
criterion, ie 
for recognizing that $C$ has finite volume, amounts to
running out of edges to walk along (Theorem~\ref{ThmComputesChamber}).  

In section~\ref{SecPreliminaries} we review some mostly-standard notation
and language. In sections~\ref{SecVinberg}--\ref{SecSlowness} we review Vinberg's
algorithm and establish the claimed slowness.
Section~\ref{SecNormedDynkinDiagrams} introduces ``normed Dynkin diagrams'', an
almost trivial variation on ordinary Dynkin diagrams.  They are
useful when finding the ``batch 0'' simple 
roots in sections~\ref{SecBatchZero} and~\ref{SecCuspidalBatch0}, at ordinary vertices
and cusps respectively, and  in edgewalking
in section~\ref{SecEdgewalking}, which also relies on our algorithms
for rank~$2$ Lorentzian lattices in section~\ref{SecShortVectors}.

\section{Preliminaries}
\label{SecPreliminaries}

\noindent
We recall some standard language and notation.  The only 
non-standard content is the notation $\Oup(L)$ and
the concepts of constraint lattices
and almost-roots.  The notation $\gend{\dots}$
means the sub\-lattice/sub\-group/\discretionary{}{}{}sub\-space generated by $\dots$,
depending on context.

A lattice $L$ means a free abelian group equipped with a
$\Q$-valued symmetric bilinear form, called the inner product
or dot product.
For $v,w\in L$, it is written $v\cdot w$, and $v^2$ is an 
abbreviation for $v\cdot v$, called the norm of~$v$.
We call $v$ isotropic if its norm is~$0$, and $L$
isotropic if it contains nonzero isotropic vectors.
$L$ is called Euclidean if it is positive definite.
If $X\sset L$ then $X^\perp$ means the
sublattice consisting of all $v\in L$ having zero inner product
with all members of~$X$.  

$L$ is called integral if  all inner products are integers.
It is called nondegenerate if $L^\perp=0$.
In this case the dual lattice $L^*=\Hom(L,\Z)$ can be
identified with the set of vectors in
$L\tensor\Q$ having integral inner products with all elements of~$L$.
When $L$ is integral and nondegenerate, then $L\sset L^*$, and
$L^*/L$ is called the discriminant group of~$L$.

$L$ is called
Lorentzian if its signature is $(n,1)$ for some $n\geq1$, meaning
that $L\tensor\R$ is the orthogonal direct sum of an $n$-dimensional
positive definite subspace and a $1$-dimensional negative definite
subspace.  In this case we make the following definitions.
A vector of positive resp.\ negative norm is called spacelike,
and a nonzero isotropic vector is called lightlike.  
The lightlike vectors
fall into two connected components, called the two light cones.   
When one of them is distinguished, we sometimes refer to its
convex hull as the future cone and call its elements
future-directed.  These terms are borrowed
from relativity.

An isometry from one lattice to another is a group isomorphism
that respects inner products.  The orthogonal group $\O(L)$ means
the group of isometries from $L$ to itself.  
Every isometry has determinant $\pm1$, and $\SO(L)$ means the index~$2$
subgroup with 
determinant~$1$.  If $L$ is Lorentzian, then
$\Oup(L)$ and $\SOup(L)$ mean the index~$2$ subgroups
of $\O(L)$ and $\SO(L)$ whose elements preserve each of the two light cones.

If $\alpha^2>0$ and 
$\alpha\cdot L\sset \frac{\alpha^2}{2}\Z$, then we call $\alpha$ an almost-root of~$L$.  
The importance of almost-roots is that the reflection $x\mapsto x-2(x\cdot \alpha)\alpha/\alpha^2$
in $\alpha$ is an element of $\Oup(L)$.  If $\alpha$ is also primitive (ie $L/\gend{\alpha}$ is torsion-free)
then we
call $\alpha$ a root.  Every almost-root is either a root or twice a root.  
The ones which are not roots are not very 
interesting, but it is convenient to
keep them around in the intermediate steps of our algorithms and then 
filter them out later.  The convenience comes from the following lemma,
which follows immediately from the definitions.

\begin{defnlemma}[Constraint lattices]
    \label{LemConstraintLattices}
    Suppose $L$ is a lattice and $N>0$.  Then the norm~$N$ vectors in
    the \emph{constraint lattice}
    $L_N=L\cap\frac{N}{2}L^*$ are exactly the norm~$N$ almost-roots of~$L$.
    \qed
\end{defnlemma}

\begin{lemma}[Root norms]
    \label{LemRootNorms}
    Suppose $L$ is a nondegenerate integral lattice and $e$ is the
    exponent of the discriminant group $L^*/L$.  
    Then the norm of any root of~$L$
    divides~$2e$.
    \qed
\end{lemma}

\begin{proof}
    Suppose $\alpha$ is a root, and  define $M=\alpha^\perp$.  By
    $L\cdot\alpha\sset\frac{\alpha^2}{2}\Z$,
    $M\oplus\gend{\alpha}$
    has index~$1$ or~$2$ in $L$.  If the index is~$1$, then 
the discriminant group $\Z/\alpha^2\Z$ of $\gend{\alpha}$
    is a direct summand of that of~$L$, so  $\alpha^2|e$.

    Now suppose the index is~$2$,
    so $L$ is generated by $M\oplus\gend{\alpha}$ and a vector
    $(u,\alpha/2)$ with $u\in M^*$.  
    Integrality forces $\alpha^2$ to be even.
    From $[L:M\oplus\gend{\alpha}]=2$
    follows $2u\in M$.
    But 
    $u\notin M$,  because otherwise $\alpha/2$ would
    lie in~$L$, contradicting the hypothesis that $\alpha$ is a root.
    So the image of $u$ in $M^*/M$ has order~$2$.
    
    Because $M$ is integral, the $\Q$-valued
    inner product on $M^*$ descends to a $\Q/\Z$-valued 
    bilinear pairing on $M^*/M$.  It is standard that this
    is nondegenerate.  Since $u\notin M$,
    there exists $m\in M^*$ with $m\cdot u\not\cong0$ mod~$1$.
    Since $u$ has order~$2$ (mod~$M$), the only
    possibility is $m\cdot u\cong\frac12$ 
    mod~$1$.
    Therefore $\lambda=(m,\alpha/\alpha^2)$ has integral inner product
    with~$(u,\alpha/2)$.  Since $\lambda$ also has integral inner
    products with $\alpha$ and all elements of~$M$, we have $\lambda\in L^*$.
    Considering its second component shows that 
    $n\lambda$ 
    can only lie in~$L$ if $\frac{\alpha^2}{2}|n$.
    Therefore the exponent of $L^*/L$ is divisible by $\alpha^2/2$.
    That is, $\alpha^2|2e$.
\end{proof}

We defined reflections in roots (and almost-roots) above.  One can 
define them more generally, for example allowing reflections in timelike vectors.
But for us, reflections mean reflections in roots.  A reflection group
means a group generated by reflections.  As a special case, the group $W(L)$
generated by all reflections of~$L$ is called the Weyl group of~$L$, and lies
in $\Oup(L)$.  

We write $\R^{n,1}$ for $(n+1)$-dimensional Minkowski space; $L\tensor\R$
is isometric to $\R^{n,1}$ for any Lorentzian lattice~$L$.  Hyperbolic
space $H^n$ means the image in $P\R^{n,1}$ of the timelike vectors, and its
boundary $\partial H^n$ means the image of the lightlike vectors.  
When we have one light cone distinguished, we sometimes identify $H^n$ with
the norm~$-1$ hyperboloid in the future cone.  In that case, we use 
future-pointing vectors to represent points of $\overline{H^n}:=H^n\cup\partial H^n$.  

The mirror of a reflection
means its fixed-point set in
$H^n$.  Removing all mirrors from $H^n$ leaves a disconnected set.  The closure
of any one of them in $\overline{H^n}$ is called a Weyl chamber (or just a chamber).  
In particular, chambers contains their ideal vertices.
When we have a chamber in mind,
we usually write $C$ for it.  The Weyl group $W(L)$ acts
simply transitively on the chambers.  
The dihedral angles of $C$ are 
integral submultiples of~$\pi$; in particular $C$ has no obtuse
dihedral angles.    Sometimes we pass without comment 
between $C$ and its preimage in
 the future cone.

Because $W(L)$ is discrete in $\Oup(L\tensor\R)$, its mirrors form a locally
finite subset of $H^n$, 
so there is subtlety when speaking of faces of~$C$ of dimension${}>1$.
There is a little subtlety when speaking of vertices, if 
$C$ has infinitely many facets.  This is only relevant in
section~\ref{SecEdgewalking}, where we give a precise definition of ``corners''.

A root $\alpha$ is called a simple root of a chamber~$C$ 
if it is orthogonal to a facet of~$C$ and has negative inner product
with vectors in the interior of~$C$.  Another way to say this is that 
$C$ is the 
projectivization of the set of future-directed vectors that have negative
inner product with all simple roots.
(The convention for Lorentzian lattices is the opposite of the
convention in Lie theory, where simple roots are inward-pointing not
outward-pointing.)
A standard consequence of the absence of obtuse dihedral angles
is that the simple roots of a (chamber of a) positive-definite
lattice are linearly independent.

\section{Vinberg's algorithm}
\label{SecVinberg}

\noindent
In this section we describe Vinberg's algorithm \cite[\S3.2]{VinbergOnTheGroups}, at first
geometrically, and then algebraically
in the special case of
reflection groups of lattices. Suppose $W$ is a discrete group
of isometries of $H^n$, generated by reflections.  In this
general setting,  the two unit vectors in~$\R^{n,1}$
orthogonal to
the mirror of a reflection are called the roots of that reflection.
Suppose 
$k\in\R^{n,1}$ 
is timelike; it is called the ``control vector'', and
we use the same symbol
for its image in hyperbolic space $H^n$.  
Suppose known a set of simple roots $B_0$ (``batch 0'')
for its $W$-stabilizer $W_k$.  Let $C_k$ be the corresponding
chamber of $W_k$, ie the cone in $\R^{n,1}$
of vectors having nonpositive inner products with the batch~$0$
roots.  There is a unique chamber $C$ of $W$ that lies in~$C_k$
and contains~$k$.  
Vinberg's algorithm extends $B_0$ to the set $B$ of simple roots for~$C$.

The construction is inductive, so for each 
$D\geq0$ we write $B_D$ (``the distance $D$ batch'') for the set 
of simple roots for~$C$ whose mirrors lie at hyperbolic distance~$D$
from~$k$.  Obviously  $B=\cup_{D\geq0}B_D$,
so it will suffice to find each~$B_D$.
We assumed $B_0$ known, which allows the induction to begin.

Suppose we have already found $B_{\leq D}=\cup_{d\leq D}B_{d}$.
Consider
all mirrors of~$W$ lying at distance${}>D$ from $k$, and let $D'$
be the distance to the closest of these.  ($D'$ exists because $W$ is discrete.)
For each mirror at that distance, consider its
outward-pointing root $\alpha$ (ie, $\alpha\cdot k<0$). We call $\alpha$ ``accepted''
if it has nonpositive inner products with 
all roots in~$B_{\leq D}$.   
Vinberg's algorithm amounts to the statement that $B_{D'}$ is this
set of accepted roots.  
(Of course, all batches between $D$ and
$D'$ are empty.)  

\medskip
The algorithm
might not terminate.  This is in the nature of things, because some
Weyl chambers have infinitely many facets.  
But even in this case, $B=\cup_{D\geq0}B_D$.
Also, in \cite[Prop.~5]{VinbergOnTheGroups}
Vinberg established a criterion on $B_{\leq D}$:  if the polytope they
define has finite hyperbolic volume,  then
all later batches are empty. 
In this case we regard the algorithm as terminating, with 
$B=B_{\leq D}$.
When one suspects that there are infinitely many simple roots,
some  ingenuity is needed to establish this rigorously.  
This is beyond the scope of this paper;   the
main method is to show that $C$ has infinitely many isometries;
see \cite[Theorem~5.5]{BorcherdsAutsOfLorentzianLattices}, 
\cite[II\S3]{NikulinRank3}, 
\cite[p.~20]{AllcockRank3} and 
\cite[Prop.~3.2]{ScharlauWalhorn} for examples of this.

\medskip
Another difficulty is that one can only call it an algorithm
if  one has an algorithm
for finding $D'$ and all roots $\alpha$ of~$W$ with $\alpha^\perp$ at 
distance $D'$ from~$k$.  
Both problems are straightforward in the case  of Lorentzian lattices.
Namely, suppose $L$ is  a Lorentzian lattice and $W\sset\Aut(L)$.
One scales the 
control vector  to be a primitive lattice
vector.  Also, one changes the definition of
the roots of a mirror: now we use the two primitive lattice vectors
orthogonal to it.     This is convenient and clearly has no effect on
the algorithm.

The hyperbolic distance between $k$ and $\alpha^\perp$ is
$
\sinh^{-1}(-k\cdot \alpha/\sqrt{-k^2 \alpha^2})
$.
Because $\sinh^{-1}$ is
increasing and $k^2$ is constant, considering roots~$\alpha$ in increasing order of
$D$ is the same as considering them in
increasing order of~$-k\cdot \alpha/\sqrt{\alpha^2}$,
which we call  the  priority of~$\alpha$.  
The algorithm examines roots with smaller priority first.
There are only finitely many possibilities for $\alpha^2$ (Lemma~\ref{LemRootNorms}).
For fixed~$\alpha^2$, the 
possibilities for the priority~$p$ lie in a scaled copy of~$\Z$.
For fixed $\alpha^2$ and~$p$, one can
enumerate all roots of $L$ with those parameters.
This boils down to iterating
over all the elements of the lattice $k^\perp$, 
or one of finitely many translates of it inside $k^\perp\tensor\Q$,
whose norm is specified in terms of $\alpha^2$ and~$p$.
Together these facts  allow one to find~$D'$ and $B_{D'}$, making
concrete the 
inductive step of Vinberg's algorithm. 

\medskip
All of Vinberg's geometric arguments apply perfectly well with hyperbolic
space  $H^n$ replaced by the sphere
$S^n$.  
The changes needed to the above are the following:  replace $\R^{n,1}$ by 
Euclidean space $\R^{n+1}$ and
$\sinh^{-1}$ by $\sin^{-1}$, and
assume $k\neq0$ rather than that $k$ is timelike.  

\medskip
Conway \cite{Conway} found another variation of the classical form of the algorithm.
He observed that the same algorithm works
perfectly well with $k$ allowed to be timelike, provided that
we remain in the Lorentzian lattice setting.  Two
observations are needed.  First, one uses a notion of ``distance'' to
the ideal point of~$H^n$, whose level sets are the horospheres centered
there.
Happily, 
this requires no change to the algorithm (when it is formulated in terms of
the priority).
Second, we can
no longer use the discreteness of~$W$ to conclude that $D'$ exists.
However, the possible priorities are constrained just as before,
so one can examine them all in increasing order.
When non-empty, the set of roots of $L$ with
given norm and priority is now infinite, but 
falls into finitely many orbits under the unipotent
subgroup of the $\O(L)$-stabilizer of~$k$.
Finding these orbits 
boils down to examining finitely many cosets of  $k^\perp/\gend{k}$ 
(or a sublattice of it) in 
$(k^\perp/\gend{k})\tensor\Q$.   So again the search for $D'$ and $B_{D'}$
can be made explicit.

\medskip
At several points it will be convenient to use a minor variation
on Vinberg's algorithm:

\begin{lemma}
    \label{LemVariation}
    Suppose $L$ is a Euclidean or Lorentzian lattice, $W$ is a subgroup of $\O(L)$
    that is generated by reflections, and
    $k\in L\tensor\R-\{0\}$.
    Call $k$ the control vector, and 
if $L$ is Lorentzian then
    make the extra assumption
    that 
    $k$ is timelike or lightlike.
    Suppose $C$ is a chamber of $W$ that contains~$k$, and
    $B_0$ is the set of simple roots of~$C$ that are orthogonal
    to~$k$.

    Suppose $\beta_1,\beta_2,\dots$ is a finite or infinite sequence of almost-roots
    of~$W$,
    having negative inner products with~$k$, such that
    \begin{enumerate}
        \item
            \label{LemVariationNonpositiveIPs}
            they have nonpositive inner products with all members of~$B_0$;
        \item
            \label{LemVariationAllSimpleRoots}
            they include all the simple roots of~$C$ other than those in~$B_0$;
        \item
            \label{LemVariationPriorityOrder}
            they appear in nondecreasing order
            of priority;
        \item
            \label{LemVariationNormOrder}
            those with equal priority appear in nondecreasing order of norm.
    \end{enumerate}
    Inductively define
     $\beta_i$ as ``approved'' if it has nonpositive inner product with all of its
    predecessors $\beta_{j<i}$ which were approved.  

    Then the approved $\beta_i$ are the
    simple roots of~$C$ that are not in~$B_0$.
\end{lemma}

\begin{proof}
    Fix $i\geq1$ and suppose that each $\beta_j$ with $j<i$ is approved if and only
    if it is a simple root of~$C$.   We will prove the same for $\beta_i$.

    First, supposing that $\beta_i$ is a simple root of~$C$, we will show that it
    is approved.  All the simple roots of~$C$ with (nonzero) priority less
    than than~$\beta_i$
     occur among the $\beta_{j<i}$, 
    by hypotheses \eqref{LemVariationAllSimpleRoots} and \eqref{LemVariationPriorityOrder}.  By induction they
    were approved, and no other predecessors of $\beta_i$ were.  Since $\beta_i$
    is a simple root, it has nonpositive inner products with all other
    simple roots.  Therefore it has nonpositive inner products with all
    approved predecessors.  So $\beta_i$ is approved.

    Next, supposing $\beta_i$ is a root and is approved, we will prove that it
    is a simple root of~$C$.  
    All simple roots of~$C$ with (nonzero) smaller priority occur among the
    $\beta_{j<i}$ by \eqref{LemVariationAllSimpleRoots} and \eqref{LemVariationPriorityOrder}, and were approved, by induction.
    Since $\beta_i$ is approved, is has nonpositive inner products with all of them.
  So $\beta_i$ is
    accepted by Vinberg's algorithm, and is therefore a simple root.

    Finally, supposing that $\beta_i$ is not a root, we must prove that it is 
    not approved.  Because $\beta_i$ is an almost-root but not a root, 
    $\beta_i/2$ is a root.  
    Note that $\beta_i/2$ has the same priority as~$\beta_i$, but strictly smaller norm.
    
    If $\beta_i/2$ is a simple root, then it equals some $\beta_j$, by \eqref{LemVariationAllSimpleRoots}.
    In fact it is a predecessor of $\beta_i$, by \eqref{LemVariationNormOrder}.  By
    induction it was approved.  From $\beta_i\cdot \beta_i/2>0$ follows that $\beta_i$ is
    not approved.

    On the other hand, if $\beta_i/2$ is not a simple root, then it is not
    accepted by Vinberg's algorithm.  So $\beta_i/2$ has positive inner product
    with some simple root with strictly smaller priority than $\beta_i/2$.
    By \eqref{LemVariationAllSimpleRoots} this simple root is some $\beta_j$, and considering priority
    shows that $j<i$.  By induction $\beta_j$ was approved.  But $\beta_j\cdot \beta_i$ 
    has the same sign as $\beta_j\cdot \beta_i/2>0$.  So $\beta_i$ is
    not approved.
\end{proof}

\section{Slowness}
\label{SecSlowness}

\noindent
In our generalization \cite{AllcockRank3} of 
Nikulin's classification of rank~$3$
reflective Lorentzian lattices 
\cite[I]{NikulinRank3}, we ran Vinberg's algorithm on
$204\,520$ candidate lattices.  Most candidates took a fraction of a second.
But some took many hours, examining long sequences of batches that all
turned out empty.
There is a phenomenon in the rank~$2$ case that predicts this behavior,
and persists in higher dimensions.

(The proof in \cite{AllcockRank3} examines $857$
candidates.  Later we simplified the overall logic at the cost of
testing more candidates, and improved our implementation of Vinberg's
algorithm.  The phenomenon of long sequences of empty batches was
already apparent for the original~$857$.)

The slowness of Vinberg's algorithm for some lattices
is closely related to the well-known fact that the fundamental solution
$(x,y)$ to Pell's equation $x^2-ny^2=1$ may be large, even  when $n$ is
small.  For example, if $n=106$ then 
$(x,y)=( 32\,080\,051 ,3\,115\,890)$.
The naive 
way to find $(x,y)$ is to start with the known solution
$(1,0)$, then look for solutions $(?,1)$, then for solutions $(?,2)$,
and so on.  This amounts to checking that $106\cdot1^2+1=107$ is not a square,
then that $106\cdot2^2+1=425$ is not a square, etc.  This process
takes  $3\,115\,890$
iterations to find a solution.  On the other hand, the continued fractions
method of solving
Pell's equation takes $17$ steps and can be
done by hand.

The main point of this section is that Vinberg's algorithm 
amounts to the naive approach.
To make this precise, take $L$ to have inner product matrix $\twobytwo{1}{0}{0}{-106}$,
take the control vector $k$ to be $(0,1)$, take batch~$0$ to consist
of the root $(-1,0)$, and take $W$ to be the
subgroup of $\Aut(L)$ generated by the reflections in norm~$1$ roots.
Then Vinberg's algorithm exactly matches the previous paragraph, finding
$3\,115\,889$ consecutive empty batches.
We made the artificial restriction to norm~$1$ roots in order to
make this simple statement.

For the more natural case $W=W(L)$,  theorem~\ref{ThmSlowness} 
shows that Vinberg's algorithm
will stop 
when it finds the  root $r=(41234, 4005)$ of norm~$106$.
This lies in the $4005$th batch
of norm~$106$ roots (taking into account the fact that $106$ must
divide $r\cdot k$).  
So the algorithm examines
at least this many batches before stopping.

It is easy to give examples with much worse performance.
In the rest of this section we suppose

$n$ is a square-free integer bigger than~$1$;

$\calO$ is the ring of algebraic integers in $\Q(\sqrt n)$;

$L$ is the lattice underlying~$\calO$ (at the scale where $1$ has norm~$1$);

$W=W(L)$;

$k$ (the control vector) is a fixed square root $\sqrt n$ of~$n$; 

$\alpha_0=-1$ (this constitutes batch 0); 

$C$ is the chamber of~$W$ containing $k$ and having $\alpha_0$ as a simple root;

$\alpha$ is the other simple root of~$C$.

\noindent
We will show how to find~$\alpha$  and give a lower bound
for the number of  batches Vinberg's algorithm examines before finding~$\alpha$.
The example above was the $n=106$ case, and   
some other examples appear in table~\ref{TableSlowness}.

\begin{table}
    \smaller
    \smaller
    \begin{tabular}{rrrrr}
        \multicolumn{1}{c}{$n$}
        &\multicolumn{1}{c}{$\alpha^2$}
        &\multicolumn{1}{c}{coefficient of~$1$}
        &\multicolumn{1}{c}{coefficient of~$\sqrt n$}
        &\multicolumn{1}{c}{batch number}
        \\
        \hline
2&2&2&1&1\\
3&6&3&1&1\\
        5&5&5\rlap{/2}&1\rlap{/2}&1\\
6&3&3&1&1\\
7&2&3&1&1\\
19&38&57&13&13\\
67&134&1809&221&221\\
73&73&9125&1068&2\,136\\
97&97&55193&5604&11\,208\\
193&193&24508105&1764132&3\,528\,264\\
241&241&1102388225&71011068&142\,022\,136\\
337&337&18648111017&1015827336&2\,031\,654\,672\\
409&409&2263478264165&111921796968&223\,843\,593\,936\\
601&601&3419107492676845&139468303679532&278\,936\,607\,359\,064\\
769&769&453881183125633513&16367374077549540&32\,734\,748\,155\,099\,080\\
    \end{tabular}
    \medskip
    \caption{The second simple root~$\alpha$ 
    of
    the lattice~$L$ underlying the ring of algebraic integers in $\Q(\sqrt n)$.
    Running Vinberg's algorithm, with $\sqrt n$ as control vector and $\{-1\}$ as batch~$0$,
    finds $\alpha$ in the indicated batch of roots of that norm.} 
    \label{TableSlowness}
\end{table}

\begin{theorem}[Slowness in rank~$2$]
    \label{ThmSlowness}
    In the setting above, take $u$ to be a fundamental unit for~$\calO$,
    and $v=u$ or $u^2$ according to whether $u$ has norm $1$ or~$-1$.
    Suppose $u$ is chosen so that $v$ has positive trace, and
    negative inner product with~$k$.
    Then $\alpha$ is the smallest positive multiple of $1+v$ 
    that lies in~$\calO$.
\end{theorem}

\begin{proof}
     The reflection $R$ in the root~$1$, which acts by the
    composition of Galois conjugation and multiplication by~$-1$,
    lies in $\Oup(L)$.
    By the positive trace hypothesis, $v$ lies in the same branch
    of the norm~$1$ hyperbola as~$1$.  So its multiplication
    operator $V$ also lies in~$\Oup(L)$.  Furthermore,
    by its definition in terms of~$u$,
    $v$ generates the multiplicative group of all lattice points in
    that branch of the hyperbola.  It follows that $\gend{V,R}$ 
    acts transitively on the norm~$1$ elements of~$L$. Since the
    $\Oup(L)$-stabilizer of~$1$ is trivial, it follows that
    $\gend{V,R}=\Oup(L)$.
    This group is infinite dihedral, because $R$ inverts~$V$.  
    Because $V\circ R$ is also a reflection, we have  $W(L)=\Oup(L)$.

    Because Galois conjugation exchanges $v$ and $v^{-1}$, it follows that $V\circ R$ sends
    sends $1$ to $-v$ and $v$ to $-1$.  Therefore
    $V\circ R$ negates~$v+1$, and   it follows that the smallest
    positive multiple $\alpha'$ of $v+1$ that lies in~$L$ is a root of~$L$.
Because $\alpha'$ has negative inner product with~$\alpha_0$, $\{\alpha_0,\alpha'\}$
    is a set of simple roots for the reflection group they
    generate.  We have seen that this group is all of~$W(L)$,
    so they are simple roots for~$L$.  Finally, by construction
    their chamber contains~$k$.  So $\alpha'=\alpha$.
\end{proof}

We used the PARI/GP package \cite{PARI} to work out many examples,
a few of which appear in table~\ref{TableSlowness}.  The last column says
which batch \emph{of roots of that norm} contains~$\alpha$.  
If $\calO$ has no elements with half-integral $\sqrt{n}$-component,
then this is just the $\sqrt{n}$-component of~$\alpha$.
Otherwise, the batch number is twice this.  
Vinberg's algorithm inspects at least this many batches to find~$\alpha$.
Usually it will also inspect  some
(empty) batches of roots of other norms.  So the later entries of the
table show that Vinberg's algorithm
can take a very long time.

Each of these troublesome
lattices can be embedded in higher-dim\-en\-sion\-al
Lorentzian lattices, dashing any hopes one might
have that this is $2$-dimensional phenomenon. 
For example, $(41234,\discretionary{}{}{} 4005,0,0)$
is a simple root of 
$\twobytwo{1}{0}{0}{-106}\oplus\twobytwo{1}{0}{0}{1}$.
Therefore running Vinberg's algorithm
with $(-1,0,0,0)$ as control vector takes at least $4005$ batches
to find it.  (Working out each 
batch in rank~$4$ also requires
more work than in rank~$2$.)

\section{Normed Dynkin diagrams}
\label{SecNormedDynkinDiagrams}

\noindent
There are several 
variations on the definition of a Dynkin diagram. The one
we use is the following:
a graph in which each pair
of nodes are joined by at most one bond, which may have any one of four types:
single, double, triple and heavy.  Double and triple bonds
are oriented, and heavy bonds may be oriented or not.  
(We interpret triple bonds as in Lie theory, indicating an angle $\pi/6$,
rather than the $\pi/5$ of \cite{VinbergOnTheGroups}.
Geometrically, heavy bonds indicate an angle ``$\pi/\infty$'', meaning parallelism.
For example, the bond in the affine diagram 
$\Atilde_1$ is a heavy bond.)

Suppose $L$ is a lattice, $\alpha_1,\dots,\alpha_n$ are roots of it, and any
two of them have nonpositive inner product and  positive-definite 
or positive-semidefinite span.  Then their Dynkin diagram 
is defined in the usual way: its nodes are $\alpha_1,\dots,\alpha_n$ with
$\alpha_i,\alpha_j$ joined (or not) as follows:
\begin{enumerate}
    \item
        if $\alpha_i\perp \alpha_j$ then they are not joined;
    \item
        if the inner product matrix of $\alpha_i$ and $\alpha_j$ is a rational
        multiple of 
        $$
        \begin{pmatrix} 2&-1\\ -1&2
        \end{pmatrix},
        \begin{pmatrix}2&-1\\-1&1 
        \end{pmatrix},
        \begin{pmatrix}6&-3\\-3&2 
        \end{pmatrix},
        \begin{pmatrix}4&-2\\-2&1 
        \end{pmatrix}\hbox{ resp. }
        \begin{pmatrix}1&-1\\-1&1 
        \end{pmatrix}
        $$
        then they are joined by a single bond, double bond, triple bond,
        oriented heavy bond, resp.\ unoriented heavy bond. In the
        oriented cases, the bond  points from the longer root to the
        shorter.
\end{enumerate}
Under the positive definiteness/semidefiniteness hypothesis,
these alternatives cover all cases.

A norm on a Dynkin diagram means  the assignment
of a positive number  to each node (its ``norm''),
subject to constraints corresponding to the previous paragraph:
\begin{enumerate}
    \item
        two nodes joined by a single bond, or by an unoriented
        heavy bond, have equal norms;
    \item
        if two nodes are joined by a double, triple, resp.
        oriented heavy
        bond, then the norm of the
        node at the tail of the arrow is $2$, $3$, resp.\ $4$ times
        the norm of the node at the
        tip of the arrow.
\end{enumerate}
A normed Dynkin diagram means a Dynkin diagram together with 
a norm on it.
Given $L$ and $\alpha_1,\dots,\alpha_n$ as above, the basic example of
a norm on their Dynkin diagram is got by defining the
norm of each node to be the norm of the corresponding root.
Conversely, if $\D$ is a normed Dynkin diagram, then it arises
in this way from a unique  inner product on the free abelian group generated by
the nodes of~$\D$: the norm on~$\D$ gives the norms of the
generators, and then the edges of~$\D$ determine their
pairwise inner products.

A Dynkin diagram $\D$ is called spherical if it admits a norm for
which the inner product on~$A$ is positive definite.  (An equivalent
definition is to require this for every norm.)
This is equivalent to  each component
of $\D$ having one of the classical types 
$A_n,\discretionary{}{}{}B_n,\discretionary{}{}{}C_n,\discretionary{}{}{}D_n,\discretionary{}{}{}E_6,\discretionary{}{}{}E_7,\discretionary{}{}{}E_8,\discretionary{}{}{}F_4,\discretionary{}{}{}G_2$.
We call $\D$
cuspidal if it has a norm which is positive semidefinite
but not definite.  
If $\D$ is connected, then this is equivalent to $\D$ appearing
in the well-known list of ``affine'' Dynkin diagrams $\Atilde_n,\dots$
(The nomenclature is less standard than in the spherical case; see
\cite{AllcockAffine},\cite{Kac},\cite{MoodyPianzola}.)  
If $\D$ is disconnected, then it is cuspidal  if and only if every component is
spherical or affine, with at least one affine component being
present.

Now suppose $\D$ is a normed spherical Dynkin diagram 
and $N$ is a positive number.  By a norm~$N$ extension of~$\D$
we mean an inclusion of $\D$ into a normed Dynkin diagram
$\D'$ with one more node than~$\D$, the extending node having norm~$N$.
We call the extension spherical resp.\ cuspidal if
$\D'$ is.
These are the only extensions that will be important for us.  

In our root-finding algorithms
for Lorentzian lattices, we will need to find all such
extensions of~$\D$.  The main application is to find all possibilities
for the root system at a vertex, given the root system
at an incident edge.
The classification of spherical Dynkin
diagrams shows that every norm~$N$ spherical 
extension occurs in one of the following ways,
where $\alpha$ denotes the extending node:
\begin{enumerate}
    \item
        $\alpha$ is not joined with any node of $\D$.
    \item
        $\alpha$ is singly joined with $1$, $2$ or $3$ nodes of~$\D$ of norm~$N$.
    \item
        $\alpha$ is doubly joined to/from a node of $\D$ of norm $N/2$ or $2N$, 
        and singly joined with $0$ or $1$ nodes of $\D$ of norm~$N$.
    \item
        $\alpha$ is triply joined to/from a node of~$\D$ of norm $N/3$ or $3N$.
\end{enumerate}
To find the norm~$N$ spherical extensions, one just constructs all these
extensions and discards the ones that are not spherical.  
The size of this enumeration is fairly small, for example if $\D$ is an $A_1^{20}$
diagram with all norms equal, then it has
$1+20+\binom{20}{2}+\binom{20}{3}=1351$ spherical extensions with that norm.
In our intended application,  the classification of reflective Lorentzian lattices, 
this is a worst-case scenario, because  Esselmann proved that no such lattices exist of rank${}>22$.
Typically there are just a few spherical extensions.

Enumerating
the cuspidal norm~$N$ extensions of $\D$ is similar; every one occurs in one
of the following ways:
\begin{enumerate}
    \item
        $\alpha$ is singly joined with $1$, $2$, $3$ or~$4$ nodes of norm~$N$.
    \item
        $\alpha$ is doubly joined to/from a node with norm $N/2$ or $2N$,
        and doubly joined to/from another node with norm $N/2$ or $2N$.
    \item
        $\alpha$ is doubly joined to/from one node which has norm $N/2$ or $2N$, and singly
        joined
        with $0$, $1$ or $2$ nodes of norm~$N$.
    \item
        $\alpha$ is triply joined to/from one node which has norm $N/3$ or $3N$, and singly
        joined with $0$ or $1$ nodes with norm~$N$.
    \item
        $\alpha$ is joined by an unoriented heavy bond with a node of $\D$ having norm $N$.
    \item
        $\alpha$ is joined by an oriented heavy bond to/from a node of $\D$ having norm $N/4$ or $4N$.
\end{enumerate}
One constructs all these extensions and discards the non-cuspidal ones.

\section{Batch 0 (spherical case)}
\label{SecBatchZero}

\noindent
In the setting which motivates this section, 
one is given a Lorentzian lattice~$L$ and a timelike 
lattice vector~$k$.   
To begin
Vinberg's algorithm, with $k$ as control vector, one must find
``batch 0'', meaning a 
system of simple roots for the root system
$\Pi_k$ consisting of all roots of $L$ that are orthogonal to~$k$.
Our goal in this section is to explain how to do this systematically.
We assume given enough members of $\Pi_k$ to span $k^\perp$
rationally.  
Geometrically
this corresponds to $k$ being a vertex of 
a Weyl chamber.   
The Lorentzian hypothesis is not relevant in this section;
all we need is a positive definite sublattice.

\begin{problem}[Spherical Batch~$0$]
	\label{ProbSphericalBatchZero}
	Suppose given  a lattice~$L$ and 
	 a positive-definite 
	subspace~$V$ of $L\tensor\R$.  Suppose
    given some roots of $L$ that span~$V$. 
	Find a set of simple roots for the 
	(spherical) root system consisting of all roots of~$L$ in~$V$.
\end{problem}

\begin{problem}[Spherical Batch~$0$, iterative step]
    \label{ProbSphericalBatchZeroIterative}
    Suppose $L$ is an integral lattice and $\Pi$ is its root 
    system.  Suppose  $\alpha_1,\dots,\alpha_{m}\in\Pi$ are linearly
    independent with positive definite span.
    Write $V_{m-1}$ for the real span of $\alpha_1,\dots,\alpha_{m-1}$, and similarly for $V_m$.
    Assuming that $\alpha_1,\dots,\alpha_{m-1}$ is a system of simple roots for $\Pi\cap V_{m-1}$,
    find the unique $\alpha_m'\in\Pi$ satisfying
    \begin{enumerate}
    \item
        \label{ProbIterativeSimpleExtension}
        $\alpha_1,\dots,\alpha_{m-1},\alpha_m'$ is a system of simple roots for $\Pi\cap V_m$, and
    \item
        \label{ProbIterativeSignDeterminer}
        $\alpha_m'\in V_{m-1}+\R_{>0}\alpha_m$.
    \end{enumerate}
\end{problem}

\begin{SolutionToSphericalIterative}
Write $\D$ for the 
normed Dynkin diagram
    $\D$ formed by $\alpha_1,\discretionary{}{}{}\dots,\discretionary{}{}{}\alpha_{m-1}$. 
    Write $k$ for
    the primitive lattice vector in $V_m\cap V_{m-1}^\perp$ that lies in the halfspace
$V_{m-1}+\R_{<0}\alpha_m$.

Step 1: find a list $\calN$ of positive numbers,
that contains the norms of all roots of~$L$.  
    (For example, appeal to
    Lemma~\ref{LemRootNorms}.)

Step 2: for each $N\in\calN$, find the set
$\calD_N$ of all norm~$N$ spherical extensions $\D\to\D'$ of normed 
    Dynkin diagrams.  

Step 3: for each $N\in\calN$, and each extension $\D\to\D'$ in $\calD_N$,
define $\alpha_{N,\D'}$ to be the (unique) element of $V_{m+1}+\R_{<0}k$ which
has norm~$N$ and whose inner products with $\alpha_1,\dots,\alpha_{m-1}$ are as
specified by~$\D'$.  
Call $\alpha_{N,\D'}$  a candidate if it lies in 
    the constraint lattice $L_N=L\cap \frac{N}{2}L^*$.

    Step 4: compute the priority $-k\cdot\alpha_{N,\D'}/\sqrt{N}$ of each candidate.
    Then $\alpha_m'$ is a candidate of  minimal priority, indeed
    the  (unique) one of them with smallest norm.
\end{SolutionToSphericalIterative}

\begin{proof}
    We will write $\Pi_m$ for $\Pi\cap V_m$, and similarly for $\Pi_{m-1}$.
    These root systems have rank $m$ and $m-1$ by the linear independence
    hypothesis.
    First we prove the existence and uniqueness of~$\alpha_m'$.
    Because $\alpha_1,\dots,\alpha_{m-1}$ is a set of simple roots for
    $\Pi_m\cap k^\perp=\Pi_{m-1}$, it has exactly two extensions to a system of
    simple roots for $\Pi_m$.  Of the corresponding two chambers for $\Pi_m$,
    one  contains~$k$ and 
    the other $-k$.  By its definition, $k$ has negative inner product
    with $\alpha_m$.  Therefore condition \eqref{ProbIterativeSignDeterminer} implies $k\cdot\alpha_m'<0$.
    So the Weyl chamber for $\alpha_1,\dots,\alpha_{m-1},\alpha_m'$ is the 
    one which contains~$k$.  This proves existence and uniqueness. 

    We will prove that $\alpha_m'$ is a candidate in the sense of step~3.
    Writing $N$ for $\alpha_m'^2$, we have
    $N\in\calN$ by the construction of~$\calN$. Similarly, 
    $\D\to\D'$ lies in $\calD_N$, where $\D'$ is the normed Dynkin
    diagram of $\alpha_1,\dots,\alpha_{m-1},\alpha_m'$.
    We claim $\alpha_m'=\alpha_{N,\D'}$.  These vectors both have
    norm~$N$, and they have the same inner products with
    $\alpha_1,\dots,\alpha_{m-1}$.  To prove they are equal it
    is enough to prove that the signs of
    their inner products with $k$ are equal.
    By its definition, $\alpha_{N,\D'}$ has negative inner product
    with~$k$, and we saw above that $\alpha_m'$ does too.  So 
    $\alpha_m'=\alpha_{N,\D'}$.
    In particular, $\alpha_{N,\D'}$ is a norm~$N$ root of~$L$,
    and therefore lies in the constraint lattice~$L_N$.  
    So $\alpha_m'$ is a candidate.

    On the other hand, every candidate $\beta$ is an almost-root of~$L$,
    because it is a norm $\beta^2$ vector  in the constraint lattice~$L_{\beta^2}$.
    Because it lies in $V_{m-1}+\R_{<0}k$, $\beta$ has negative inner product with~$k$.
    Furthermore, its inner products with $\alpha_1,\dots,\alpha_{m-1}$ are 
    given by a normed Dynkin diagram, and are therefore nonpositive.
    We have verified that the sequence of candidates, ordered first by priority
    and second by norm, satisfies the
    hypotheses of  Lemma~\ref{LemVariation}, with $k$ as control vector. 

    The conclusion of that lemma is that
    step~$4$ reveals $\alpha_m'$.
    In detail:
    We know that only one candidate is a simple root, and we
    already named it $\alpha_m'$.
    Therefore $\alpha_m'$ is the only  candidate approved by Lemma~\ref{LemVariation}.  
    The first candidate
    is automatically approved, so it is $\alpha_m'$.
    This holds for any ordering of the candidates satisfying the
    ordering hypotheses \eqref{LemVariationPriorityOrder}--\eqref{LemVariationNormOrder} of Lemma~\ref{LemVariation}.
    So only $\alpha_m'$ has smallest norm, among all candidates of
    smallest priority. 
\end{proof}

\begin{proof}[Solution to Problem~\ref{ProbSphericalBatchZero}]
In the
    situation of Problem~\ref{ProbSphericalBatchZero}, suppose $\alpha_1,\discretionary{}{}{}\dots,\discretionary{}{}{}\alpha_n$ are the given roots that span~$V$.
By discarding those which are linear combinations of their predecessors, 
we may suppose they are linearly independent.  For each $m$, we write $V_m$ for the
rational span of $\alpha_1,\dots,\alpha_m$.  Obviously $\alpha_1$ forms a system of simple roots
for $\Pi\cap V_1$.  For consistency with the inductive step, define $\alpha_1'=\alpha_1$.

    Now suppose $m>1$ and that $\alpha_1',\dots,\alpha_{m-1}'$ is a system of simple roots for 
    $\Pi\cap V_{m-1}$.  In particular, $\alpha_1',\dots,\alpha_{m-1}'$ have the same rational span
    as $\alpha_1,\dots,\alpha_{m-1}$, so $\alpha_1',\dots,\alpha_{m-1}',\alpha_m$ are linearly independent with
    positive definite span.  Using a solution to Problem~\ref{ProbSphericalBatchZeroIterative}, find
    the unique $\alpha_m'\in\Pi\cap(V_{m-1}+\R_{>0}\alpha_m)$, which together with $\alpha_1',\dots,\alpha_{m-1}'$
    forms a system of simple roots for $\Pi\cap V_m$.  Repeating this argument through the $m=n$
    case shows that $\alpha_1',\dots,\alpha_n'$ are a system of simple roots for $\Pi\cap V$.
\end{proof}

\begin{remark}
$\Pi\cap V$ has more than one system of simple roots, and  we seldom care
which one is used.  But the proof distinguishes  one, in 
    terms of the initial roots $\alpha_1,\dots,\alpha_n$.  One can use this to detect
    whether $\alpha_1,\dots,\alpha_n$ form a system of simple roots for $\Pi\cap V$. 
    Namely,
    they do if and only if $\alpha_m'=\alpha_m$ for all $m$.
\end{remark}

\section{Batch 0 (Cuspidal case)}
\label{SecCuspidalBatch0}

\noindent
This section has the same motivation as section~\ref{SecBatchZero}, except with
the control vector $k$ being an ideal point of hyperbolic space.
The method is conceptually similar to the iterative step of the spherical
case, but some details are more complicated. 

\begin{problem}[Cuspidal Batch 0]
    \label{ProbCuspidalBatch0}
    Suppose $L$ is a Lorentzian lattice, $\Pi$ is its root system,
    $k$ is a primitive
    lightlike lattice vector, and $V$ is its orthogonal complement
    in $L\tensor\R$.  
    Suppose $U$ is a complement to $\R k$ in $V$,
    and $u_1,\dots,u_n$ are a system of simple roots for $\Pi\cap U$ and span~$U$.
    Let $C\sset H^{n+1}\cup\partial H^{n+1}$ be the unique Weyl chamber 
    for $W(\Pi\cap V)$ that contains $k$ and whose simple roots include
    $u_1,\dots,u_n$.  Find the remaining simple roots of $C$.
\end{problem}

\begin{remark}
    Suppose that instead of $u_1,\dots,u_n$
    we are given enough members of $\Pi\cap V$ to span $V$ modulo~$\R k$,
    with no other hypotheses.  Then by discarding some of them,
    and applying the spherical batch $0$ case, we could construct
    $u_1,\dots,u_n$ with the properties stated in the problem.  
\end{remark}

\begin{SolutionToCuspidalBatch0}
    Write $\Delta$ for the normed Dynkin
    diagram of $u_1,\dots,u_n$.
    Take the future cone to be the convex hull of the light cone that
    contains~$k$.
    Choose $k'$ to be any future-directed vector in the plane~$U^\perp$,
    but not in $\R k$.  

Step 1: Find a list $\calN$ of positive numbers, that contains the
    norms of all the roots of~$L$. (For example, appeal to Lemma~\ref{LemRootNorms}.)

    Step 2: For each $N\in\calN$, find the set $\calD_N$ of all norm~$N$ cuspidal 
    extensions $\D\to\D'$ of normed Dynkin diagrams. (See Section~\ref{SecNormedDynkinDiagrams}.)

    Step 3: For each $N\in\calN$, and each extension $\D\to\D'$ in $\calD_N$,
    consider the element $u_{N,\D'}$ of $U$ whose inner products with
    $u_1,\dots,u_n$ are as specified by $\D'$.  If the
    constraint lattice $L_N=L\cap\frac{N}{2}L^*$
    contains a vector of the form $u_{N,\D'}+c k$ with $c\in\Q$, then define
    $v_{N,\D'}$ as the vector of this form with $c$ as small as possible
    subject to being positive.  Call 
     these vectors the candidates.

    Step 4: 
    Compute the priority $-k'\cdot v_{N,\D'}/\sqrt{N}$ of each candidate,
    using $k'$ as control vector.
    Order the candidates first by priority and second by norm.

    Step 5: 
    Inductively define a candidate to be approved if it has nonpositive inner products
    with all of its predecessors that were approved.   Then 
    the simple roots of~$C$ are $u_1,\dots,u_n$ and the approved candidates.
\end{SolutionToCuspidalBatch0}

\begin{proof}
    Because $u_1,\dots,u_n$ are among the simple
    roots of~$C$,  and $k'$ is orthogonal to them, 
    $k'$ lies in an edge of~$C$.  So all remaining simple
    roots of~$C$ have negative inner products with~$k'$.

    We begin by claiming that every simple root $\alpha$ of~$C$, other than $u_1,\dots,u_n$, is
    a candidate in the sense of step~3.  First, $\alpha\notin U$, because $u_1,\dots,u_n$ are already
    a set of simple roots for the $n$-dimensional spherical root system
    $\Pi\cap U$.  
    Therefore $u_1,\dots,u_n,v$ are a basis for~$V$.
    Second, write $N$ for $\alpha^2$; by definition $\calN$ contains $N$.  
    Third, write $\D'$ for the 
    normed Dynkin diagram got by adjoining $\alpha$ to $\D$.  Because
    $\D'$ forms a basis for $V$, which is positive semidefinite but not
    positive definite, $\D\to\D'$ is a norm~$N$ cuspidal extension.
    Therefore $u_{N,\D'}$ is defined.  Its inner products with~$u_1,\dots,u_n$
    are determined by the diagram extension, so they are the same as the
    inner products of $\alpha$ with $u_1,\dots,u_n$.  It follows that $u_{N,\D'}$
    is the projection of $\alpha$ to~$U$.  

    Therefore $\alpha=u_{N,\D'}+ck$ for
    some $c\in\Q$.  
    From $\alpha\cdot k'<0$, $k\cdot k'<0$ and
    $\alpha\cdot k'=ck\cdot k'$, 
    we see $c>0$.
    We have $\alpha\in L_N$ because  $\alpha$ is a norm~$N$ root. So
    the candidate $v_{N,\D'}$ is defined.  
    We claim that $\alpha$ coincides with it.  By definition, $v_{N,\D'}\in L_N$ has the
    form $u_{N,\D'}+c_0k$ with $0<c_0\leq c$.  
    So it is enough to prove $c_0=c$.  We suppose $c_0<c$
    and derive a contradiction.
    Since $v_{N,\D'}$
    is a norm~$N$ vector in $L_N$, it is an almost-root, so
    it is either a root or twice a root.  In either case,
     the mirror $v_{N,\D'}^\perp$ strictly separates $k'$ from
    $\alpha^\perp$.
    Therefore no chamber having $\alpha$ as a simple root can 
    contain~$k'$, which is a contradiction.

    Next we observe that every candidate $\beta$ is an almost-root,
    since it is a norm $\beta^2$ element of $L_{\beta^2}$.  By
    the $c>0$ condition in the definition of $\beta$, we have $k'\cdot\beta<0$.
    The inner products of $\beta$ with $u_1,\dots,u_n$ are given by
    the normed Dynkin diagram $\D'$ and are therefore nonpositive.
    Now Lemma~\ref{LemVariation}, applied
    to the root system $\Pi\cap V$, finishes the proof.
\end{proof}

\section{Short vectors in 2-dimensional Lorentzian lattices}
\label{SecShortVectors}

\noindent
Throughout this section, take $L$ to be a $2$-dimensional Lorentzian
lattice and $k\in L\tensor\Q$ to be timelike or lightlike.  Write
$P$ for 
the plane  $L\tensor\R$, and suppose $\frac12P$ is
an open halfplane in $P$ bounded by $\R k$, that contains
timelike vectors arbitrarily close to~$k$.  
We think of $k$ as
defining a point (ordinary or ideal) of hyperbolic $1$-space $H^1$,
and $\frac12P$ as representing a ray emanating from it.  The condition
on the existence of spacelike vectors near $k$ corresponds to the
fact that there is only one ``ray'' emanating from an ideal point of
$H^1$, rather than two from an ordinary point.  The  
goal of this section is to search for short
spacelike lattice vectors $r\in\frac12P$
whose orthogonal complements are near~$k$.  

We write $S$ for the sector of vectors in $\frac12P$ with nonnegative norm.
We illustrate the situation in figure~\ref{FigPlane}, according to
whether $k$ is timelike or lightlike.
We order the rays emanating from $0$ into~$S$ 
according
to their slope in the figure (arrows indicate the direction
of increase).  This induces a partial order on~$S$
and a total order on the 
set of primitive lattice vectors that lie in~$S$.  
We write $\Omega$ for the set of lightlike vectors along the top of~$S$.
Take $r$ to
be the primitive lattice vector in~$S$ that is orthogonal to~$k$.
In section~\ref{SecEdgewalking}, $L$ will be the projection of a higher-dimensional
lattice to a plane~$P$, and $r$ will be got by projecting
a root of it to~$P$ and then scaling to make it primitive.
Section~\ref{SecEdgewalking}
is also the source of the label ``ray/line to walk along''.

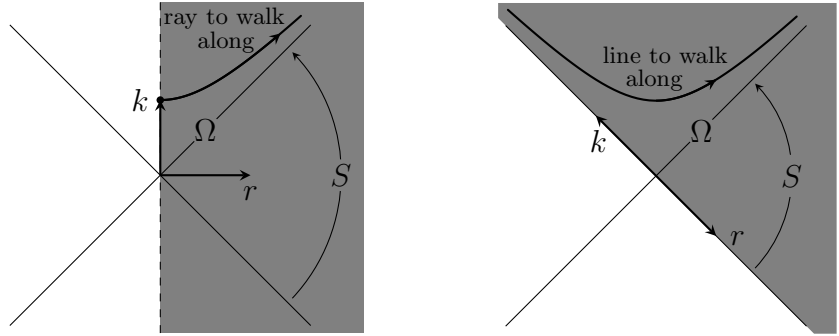
\begin{figure}
\def\Shaded{gray}
\begin{tikzpicture}[x=1.0cm,y=1.0cm]
	\def\InnerRadius{1.5}
	\def\OuterRadius{2.4}
	\fill[fill=\Shaded](2.7,-2.1)rectangle(0,2.3);
	\draw(-2,-2)--(2,2);
	\draw(2,-2)--(-2,2);
	\fill[fill=\Shaded](.6,.6) circle(6pt);
	\draw(.6,.6) node{$\Omega$};
	\draw[dashed](0,-2.1)--(0,0);
	\draw[dashed](0,2.3)--(0,0);
	\draw[thick,-stealth](0,0)--(0,1) node[anchor=east]{$k$};
	\draw[thick,-stealth](0,0)--(1.2,0) node[anchor=north]{$r$};
	\draw[thick,domain=0:1.6,samples=50,-stealth] plot (\x,{sqrt(1+\x*\x)});
	\draw[cap=round,thick,domain=0:1.87,samples=50] plot (\x,{sqrt(1+\x*\x)});
	\fill[fill=black](0,1)circle(.05);
	\draw(1.8,1.8) node[anchor=south east]{\scriptsize ray to walk};
	\draw(1.4,1.8) node[yshift=-9pt,anchor=south east]{\scriptsize along};
	\draw[-stealth](-43:\OuterRadius) arc(-43:43:\OuterRadius);
	\fill[fill=\Shaded](0:\OuterRadius) circle(6pt);
	\draw(0:\OuterRadius) node{$S$};
\end{tikzpicture}
\qquad\qquad
\begin{tikzpicture}[x=1.0cm,y=1.0cm]
	\def\Radius{1.8}
	\fill[fill=\Shaded](2.4,-2.1)--(2.4,2.3)--(-2.1,2.3)--(-2.1,2.1)--(2.1,-2.1)--cycle;
	\draw(2,2)--(-2,-2);
	\draw(2,-2)--(-2,2);
	\fill[fill=\Shaded](.6,.6) circle(6pt);
	\draw(.6,.6) node{$\Omega$};
	\draw[thick,-stealth](0,0)--(-.8,.8) node[xshift=1pt,yshift=-1pt,anchor=north]{$k$};
	\draw[thick,-stealth](0,0)--(.8,-.8) node[xshift=1pt,yshift=-1pt,anchor=west]{$r$};
	\draw[thick,domain=0:.8,samples=50,-stealth] plot (\x,{sqrt(1+\x*\x)});
	\draw[cap=round,thick,domain=-1.97:1.87,samples=50] plot (\x,{sqrt(1+\x*\x)});
	\draw(1.1,1.3) node[anchor=south east]{\scriptsize line\llap{\phantom y} to walk};
	\draw(.5,1.3) node[yshift=-9pt,anchor=south east]{\scriptsize along};
	\draw[-stealth](-43:\Radius) arc(-43:43:\Radius);
	\fill[fill=\Shaded](0:\Radius) circle(8pt);
	\draw(0:\Radius) node{$S$};
\end{tikzpicture}
\caption{The plane $P=L\tensor\R$, with timelike or lightlike vertex~$k$.  
	The diagonal lines are lightlike, the shading indicates~$\frac12P$, 
	and the circular arcs indicate the ordering on the sector $S$.}
\label{FigPlane}
\end{figure}

\begin{problem}
    \label{ProbShortVectorsInOrder}
    Given a primitive lattice vector $r\in S-\Omega$, 
    and $M>0$, find the primitive lattice vectors of
    norm${}\leq M$ in~$S$ that come after~$r$, in increasing order.
\end{problem}

Algorithm~\ref{AlgPromised} below contains the essential idea of the solution.
The name ``{\sf Promised}'' refers to the hypothesis that $S$ contains
a primitive lattice vector of norm${}\leq M$, after $r$.  There are
two natural settings in which this is automatic.  First, if $L$
is isotropic, then the primitive lattice vector in~$\Omega$
comes after $r$ (by the hypothesis $r\notin\Omega$) and has norm~$0$.
Second, if $L$ is anisotropic and $r^2\leq M$, then 
$S$ contains the images of $r$ under $\SOup(L)\iso\Z$, half of
which come after~$r$.

We distinguish an orientation on $P$ by taking a basis whose
first member lies in the interior of $\frac12P$ and whose second
member is~$k$.  An ``oriented basis'' will mean a basis for~$L$
representing this orientation.

Suppose $M>0$ and $r\in S$ is a primitive lattice vector. Then any
$l\in L$ is called an $M$-supplement of $r$ if two conditions
hold.  First, $(r,l)$ must be an oriented basis.  Second, $l$ must not
lie in the interior of the convex hull of the norm~$M$ hyperbola in~$S$.
We think of $r$ as standing for ``right'' because it points into~$S$,
and $l$ for ``left''.  
If the first condition holds, then by subtracting a multiple of~$r$ from~$l$
we can arrange for the second condition to also hold.

\begin{algorithm}[{\sf Promised($M$,$r$,$l$)}]
    \label{AlgPromised}
    Suppose 
    $r\in S-\Omega$ is a primitive lattice vector, $M>0$,
    and $l$ is an $M$-supplement of~$r$.  Suppose also that $S$
    contains a primitive lattice vector of norm${}\leq M$ that
    comes after~$r$.  
    
    Then this recursive algorithm returns a pair $(r',l')$, where $r'$ is the
    first such vector,
    and $l'$ is
    an $M$-supplement of~$r'$.

	\begin{enumerate}
		\item\relax{\sf[Find middle]}
			\label{StepFindMiddle}
			Set $m\leftarrow r+l$.
		\item\relax{\sf[Go right?]}
			\label{StepRight}
			If $m^2\leq M$ or $m\cdot l<0$ then return {\sf Isotropic($M$,$r$,$m$)}.
		\item\relax{\sf[Done?]}
			\label{StepDone}
			If $l^2\geq0$ and $r\cdot l>0$
			then return $(l,-r)$.			
		\item\relax{\sf[Go left]}
			\label{StepLeft}
			Return {\sf Isotropic($M$,$m$,$r$)}.
	\end{enumerate}
\end{algorithm}

We first remark that $r'$ exists.  By hypothesis, $S$ contains
a primitive lattice vector of norm${}\leq M$ after~$r$. The set
of such vectors is discrete in~$P$, 
and it follows easily that  there is a first one.

We also remark that in the case of $L$ isotropic, this algorithm
suffices to solve Problem~\ref{ProbEdgewalking}.   Given~$r$, one finds
an $M$-supplement for it, and then repeats the algorithm
finitely many times, using the output of each iteration
as input for the next.  The process stops once the primitive lattice vector in~$\Omega$
is found.

\begin{proof}
    First we note that $r'$ lies in the open half-plane $\R r+\R_{>0}l$, because
    it comes after $r$ in the ordering on~$S$.  We claim that  $r'$ even
    lies in the half-open sector $\R_{\geq0}r+\R_{>0}l$.  If
    $l^2<0$ or $l\cdot r<0$, then the intersection of the half-plane with~$S$
    lies in this sector, so certainly $r'$ does too.
    On the other hand, if $l^2\geq0$ and $l\cdot r\geq0$, then $l$ lies in~$S$, and it
    has norm${}\leq M$
    by the $M$-supplement hypothesis.  By definition, $r'\leq l$ with respect to
    the ordering on~$S$,
    so again $r'$ lies in 
    $\R_{\geq0}r+\R_{>0}l$.

    The proof that the algorithm terminates, and gives the
    claimed results, uses induction on the sum of the 
    (just-proven-to-be nonnegative) coefficients of $r'$ with respect to $r,l$.
    The base of the induction is step~\ref{StepDone}.
For orientation  we remark that
the primitive lattice vector
$m$ lies in the ``middle'' of this sector,
    and  subdivides it into the ``right'' subsector 
$\R_{\geq0}r+\R_{>0}m$
and
``left'' subsector $\R_{\geq0}m+\R_{>0}l$.
The  ``go right'' and ``go left'' steps search for $r'$ in these
subsectors.  

First suppose step~\ref{StepRight} applies. Then $r'$ lies in 
    the right subsector $\R_{\geq0}r+\R_{>0}m$, by the same argument we used
    to prove  $r'\in\R_{\geq0}r+\R_{>0}l$.
    Furthermore, 
    the conditions for step~\ref{StepRight} say
    that $m$ is an $M$-supplement
    of~$r$, so $M,r,m$ satisfy the hypotheses of
    {\sf Isotropic($\cdot$,$\cdot$,$\cdot$)}.  
    The $m$-coefficient of $r'$ with respect to $(r,m)$ is the same as the
    $l$-coefficient of $r'$ with respect to $(l,m)$.  And the $r$-coefficient
    of $r'$ with respect to $(r,m)$ is strictly smaller than with respect to $(r,l)$.
    By induction, 
    {\sf Isotropic($M$,$r$,$m$)} returns $(r',l')$ where $l'$ is an $M$-supplement 
    to~$r'$.  This completes the proof if step~\ref{StepRight} applies.

    Now suppose step~\ref{StepDone} applies.
    It will be enough to prove $r'=l$, because $(l,-r)$ is an oriented basis,
    and $-r$ is an $M$-supplement to~$l$ by $-r\notin S$.  
    By the conditions for step~\ref{StepDone}, $l$ lies in~$S$, so the
    $M$-supplement hypothesis forces $l^2\leq M$.  
    Because $r'\in \R_{\geq0}r+\R_{>0}l$,  to show $r'=l$ it suffices to
    show that every $v\in\Z_{>0}r+\Z_{>0}l$ has norm${}>M$.
    Note that $v$ can be expressed as the sum of $m$ and some nonnegative
    multiples of $r$ and~$l$.   Therefore $v^2>M$ follows from the facts
    $r^2\geq 0$ (since $r\in S$),
    $l^2\geq0$ and $l\cdot r>0$ (since this step does apply), 
    $m^2>M$ (since step~\ref{StepRight} didn't apply), and $m\cdot l\geq0$ and $m\cdot r\geq0$
    (since $m=r+l$).
    This completes the proof if step~\ref{StepDone} applies.

    Finally, suppose step~\ref{StepLeft} applies.  The calculation in the
    previous paragraph shows that every $v\in\Z_{\geq0}r+\Z_{>0}m$
    has norm${}>M$.  Therefore $r'$ cannot lie in the right subsector,
    and must lie in the left subsector $\R_{\geq0}m+\R_{>0}l$.
    Since step~\ref{StepRight} didn't apply, $m$ lies in~$S$, but not in $\Omega$.  
    Obviously
    $(m,l)$ is an oriented basis for~$L$.  Since $l$ is an $M$-supplement
    to~$r$, it is also one for~$m$.  Therefore $(M,m,l)$ satisfy the
    hypotheses of ${\sf Isotropic(\cdot,\cdot,\cdot)}$.  
    Arguing as for step~\ref{StepRight} shows that the coefficients
    of $r'$ with respect to $(m,l)$ have smaller sum than do its 
    coefficients with respect to $(r,l)$.  By induction,
    ${\sf Isotropic}(M,m,l)$ returns $(r',l')$ where $l'$
    is an $M$-supplement for~$r'$.  This completes the proof.
\end{proof}

\begin{remark}[Performance]
    \label{RkPerformance}
    Write $(r_n,l_n)$ for the second and third arguments to the
    $n$th call to
    {\sf Promised}, starting with $(r_0,l_0)$.  Write $\twobyone{a_n}{b_n}$
    for the coefficients of the final answer $r'$ with respect
    to $(r_n,l_n)$.  Because
    $(r_{n+1},l_{n+1})=(r_n,r_n+l_n)$ or
    $(r_n+l_n,l_n)$,  $\twobyone{a_{n+1}}{b_{n+1}}$ 
    is got from $\twobyone{a_n}{b_n}$ by 
left-multiplying by 
     $\twobytwo{1}{-1}{0}{1}$ or $\twobytwo10{-1}1$.
    All coefficients arising this way are positive.
    Therefore the algorithm amounts to running the 
    ``subtract the smaller from the larger'' version of the
    Euclidean algorithm
    on the coefficients $\twobyone{a_0}{b_0}$.
    Of course, we don't know
    ahead of time what these are.  Running the Euclidean algorithm
    backwards  amounts to computing the continued fraction
    expansion of $a_0/b_0$.   
    This is the sense in which
    our alternative
    to Vinberg's algorithm resembles the continued fraction expansion
    approach to Pell's equation.

    A simple modification of our algorithm makes it analogous
    to the usual ``replace the larger by its remainder upon division
    by the smaller'' version of the Euclidean algorithm.
    We group together consecutive applications of step~\ref{StepRepeat}, by
     calling ${\sf Promised}(M,r,{\sf Canonical}(M,r,l))$
    instead of ${\sf Promised}(M,r,r+l)$. Here
    ${\sf Canonical}(M,r,l)=l+?r$ is the canonical $M$-supplement
    discussed below.  We group together consecutive applications
    of step~\ref{StepLeft} in a similar way.

    It is standard, but usually phrased differently,
    that the size of the arguments to
    the (usual) Euclidean algorithm grows exponentially in the number of
    steps the algorithm takes.  
    The same reasoning shows that the time required to write down 
    the result of our modified Algorithm~\ref{AlgPromised} (in 
    decimal notation) grows at least linearly in  
    the number of steps taken.
    This does not prove that the
    run time is linear in the size of the answer,
    because as the numbers grow larger, the arithmetic operations 
    in each step take longer.  But it does suggest that
    Algorithm~\ref{AlgPromised} will perform very quickly.  Since the
    rest of the algorithms in this section, and our alternative
    to Vinberg's algorithm (section~\ref{SecEdgewalking}), are just wrappers
    around it, they should be similarly fast.

    See
    \cite[\S5.7.1--5.7.2]{Cohen} and references there 
    for more precise estimates of the running time
    of the continued fraction algorithm, taking into account the growth
    of the coefficients.  See also the survey \cite{Lenstra} for 
    other approaches to Pell's equation.
\end{remark}

In the rest of this section we suppose $L$ is anisotropic.  The main
complication is that $\SOup(L)\iso\Z$, and our
algorithms must search for a generator at the same time they
search for vectors of small norm.  To do this we introduce
what we call the canonical $M$-supplement to a primitive 
lattice vector $r\in S$, for a given $M>0$.  
This is the unique $M$-supplement $l$ for~$r$, for which
$l+r$ is not an $M$-supplement for~$r$.  
The point of this construction is that 
$\SOup(L)$ respects canonical 
$M$-supplementation: if $\phi\in\SOup(L)$,
then the canonical $M$-supplement of $\phi(r)$ is the
$\phi$-image of the canonical $M$-supplement of~$r$.

To find the
canonical $M$-supplement, start with any $M$-supplement and
repeatedly add $r$ to it.  Alternately, one can give
a formula for it:

\begin{lemma}[{\sf Canonical($M$,$r$,$l$)}]
	\label{LemmaCanonical}
	Suppose $L$ is integral and anisotropic, $M$ is a positive integer and 
	$(r,l)$ is an oriented basis for $P$ with $r\in S$.   
    Then the canonical $M$-supplement of~$r$ is
	$l+Kr$ where
	\begin{displaymath}
		K=
			\floor[\Big]{\frac{-r\cdot l+\floor[\big]{\sqrt{(r\cdot l)^2-r^2(l^2-M)}}}{r^2}}
	\end{displaymath}
    \qed
\end{lemma}

We explained above that  if $r^2\leq M$, then ${\sf Promised}$
is enough to find the next primitive lattice vector $r'\in S$
of norm${}\leq M$.  But if $r^2>M$, then such a vector
might not even exist.  To approach the problem we start
with a more modest goal: finding a shorter spacelike vector,
or determining that none exist.

\begin{algorithm}[{\sf Shorter($r$,$l$)}]
    \label{AlgorithmShorter}
	Suppose $L$ is anisotropic, $r\in S$ is a primitive lattice vector,  and $l$ is
	its canonical $r^2$-supplement.  

    If $L$ has no spacelike vectors
    of norm${}<r^2$, then this algorithm says so.
    Otherwise it returns $(r',l')$, where $r'$
    is the first primitive such vector of~$S$ after~$r$,
    and $l'$ is the canonical $(r')^2$-supplement of~$r'$.

	\begin{enumerate}
		\item\relax{\sf[Constants]}
			\label{StepConstants}
            Set $M\leftarrow r^2$ and then $A\leftarrow\twobytwo{M}{r\cdot l}{r\cdot l}{l^2}$.
		\item\relax{\sf[Enter main loop]}
			\label{StepMainLoop}
			Set $(r,l)\leftarrow{}${\sf Promised($r^2$,$r$,$l$)}
        \item\relax{\sf[Canonicalize]}
            \label{StepCanonicalize}
            Set
			$l\leftarrow{}${\sf Canonical($r^2$,$r$,$l$)}.
		\item\relax{\sf[Found?]}
			\label{StepFound}
			If $r^2<M$ then return $(r,l)$.
		\item\relax{\sf[Nonexistence?]}
			\label{StepNonexistence}
			If the inner product matrix of $(r,l)$ is equal to~$A$, then stop: $L$ contains no
			spacelike vectors of norm${}<M$.
		\item\relax{\sf[Repeat]}
			\label{StepRepeat}
			Go back to step~\ref{StepMainLoop}.
	\end{enumerate}
\end{algorithm}

\begin{proof}
	We write $(r_n,l_n)$ for the value of $(r,l)$ at the $n$th entry to step~\ref{StepMainLoop}.
	We will prove the following inductively.  First, $r_n$ lies in $S$ and has norm~$M$, 
  and $l_n$ is its canonical $M$-supplement.  Second, $0$ is the only
    lattice vector of norm${}<M$ in the
    sector $\R_{\geq0} r_1+\R_{\geq0} r_n$.
	These properties are trivial for $(r_1,l_1)$.  Now suppose $n\geq1$; we will work through
    the steps and see what the algorithm does.

	Step~\ref{StepMainLoop} makes sense because $S$ contains vectors of norm $r_n^2$ that
    come after $r_n$, 
    for example some of the $\SOup(L)$-images of $r_n$.  After this step and the next,
    $r$ is the first vector of norm${}\leq M$ in $S$ after~$r_n$, and $l$ is the 
    canonical $r^2$-supplement to $r$.

    If step~\ref{StepFound} applies, then by induction, $r$ is the first vector of
    norm${}<M$ in $S$ after $r_1$, and we already know that $l$ is its canonical
    $r^2$-supplement.  So the algorithm halts with the claimed result.  So suppose
    step~\ref{StepFound} does not apply.  That is, $r^2=M$.  By induction,
    $0$ is the only lattice vector of norm${}<M$ in the sector $\R_{\geq0}r_1+\R_{\geq0}r$.

    If step~\ref{StepNonexistence} applies, then $(r,l)$ has the same inner product
    matrix as $(r_1,l_1)$, so the linear transformation sending $(r_1,l_1)$ to $(r,l)$
    is a member of $\SOup(L)$.
    Therefore $S$ is the union of the $\SOup(L)$-images of the sector $\R_{\geq0}r_1+\R_{\geq0}r$,
    which we just saw  contains no lattice vectors of norm${}<M$ except for~$0$.
    So the same holds for~$S$, and  the algorithm halts with the claimed result.

    If step~\ref{StepNonexistence} does not apply, then we return to step~$2$,
    so the pair $(r_{n+1},l_{n+1})$ is defined.  We have already verified our
    inductive hypotheses on it.  It remains only to show that the algorithm terminates.
    Otherwise we would have an infinite sequence $(r_1,l_1),(r_2,l_2),\dots$,
    where the $r_n$ are consecutive norm~$M$ lattice vectors in~$S$, and each $l_n$
    is the canonical $M$-supplement to $r_n$.  Because there are only finitely
     many $\SOup(L)$-orbits of norm~$M$
    lattice vectors, some $g\in\SOup(L)$ sends $r_1$ to some $r_{n>1}$.  By the nature
    of the canonical $M$-supplement, $g$ sends $l_1$ to $l_n$.  So $(r_n,l_n)$
    has inner product matrix~$A$, and the algorithm would have halted via step~\ref{StepNonexistence}. 
\end{proof}

\begin{algorithm}[{\sf NotPromised($M$,$r$,$l$)}]
    \label{AlgorithmNotPromised}
    Suppose $L$ is anisotropic, $r\in S$ is a primitive lattice vector, 
    $l$ is the canonical $r^2$-supplement to~$r$, and 
    $M>0$.  

    If $L$ has no spacelike vectors of norm${}\leq M$, then this algorithm says so.
    Otherwise, it returns $(r',l')$ where $r'$ is the first primitive such
    vector in~$S$ after $r$, and $l'$ is the canonical $M$-supplement of~$r'$.

    \begin{enumerate}
        \item\relax{\sf[Easy case]}
            \label{StepNotPromisedEasyCase}
            If $r^2\leq M$, then set $(r,l)\leftarrow{\sf Promised}(M,r,l)$
            and then $l\leftarrow{\sf Canonical}(M,r,l)$, and return
            $(r,l)$.
        \item\relax{\sf[Nonexistence?]}
            \label{StepNotPromisedNonexistence}
            Set $(r,l)\leftarrow{\sf Shorter}(r,l)$, unless ${\sf Shorter}(r,l)$
            reports nonexistence.  In the latter case, stop: $S$ contains
            no vectors of norm${}\leq M$.
        \item\relax{\sf[Done?]}
            \label{StepNotPromisedDone}
            If $r^2\leq M$, then set $l\leftarrow{\sf Canonical}(M,r,l)$
            and return $(r,l)$.
        \item\relax{\sf[Repeat]}
            \label{StepNotPromisedRepeat}
            Go back to step~\ref{StepNotPromisedNonexistence}.
    \end{enumerate}
\end{algorithm}

\begin{proof}
    If case~\ref{StepNotPromisedEasyCase} applies, then $S$ contains primitive lattice
    vectors of norm${}\leq M$ after~$r$, for example some of the
    $\SOup(L)$-images of~$r$.  The only hypothesis of ${\sf Promised}(\cdot,\cdot,\cdot)$
    that requires checking is that $l$ is a not-necessarily-canonical $M$-supplement of~$r$.
    This follows from the fact that it is an $r^2$-supplement and
    $r^2\leq M$.
    We refer to {\sf Promised} to justify the result in this case.
    ({\sf Promised} 
    does not 
    canonicalize
    the $M$-supplement in its output, which is why we do it here.)
    So suppose case~\ref{StepNotPromisedEasyCase} does not apply.  That is, $r^2>M$.

    We write $(r_n,l_n)$ for the value of $(r,l)$ at the $n$th entry
    to  step~\ref{StepNotPromisedNonexistence}.  We will prove by induction that the $r_i$
    form an increasing sequence of primitive lattice vectors
    in~$S$, the norms of the $r_i$
    exceed~$M$ and form a strictly decreasing sequence, and that there
    are no lattice vectors of norm${}< r_n^2$ except~$0$
    in the sector 
    $\R_{\geq0}r_1+\R_{\geq0}r_n$.  When $n=1$ these claims are
    trivial.  Now suppose $n\geq1$; we will work through the steps and
    see what the algorithm does.

    In step~\ref{StepNotPromisedNonexistence}, if ${\sf Shorter}(r_n,l_n)$ reports nonexistence,
    then $S$ contains no lattice vectors of norm${}<r_n^2$, hence none of norm${}\leq M$,
    as claimed.  Otherwise, $r$ is the first primitive lattice vector
    after $r_n$ in~$S$ with norm${}<r_n^2$, and $l$ is the 
    canonical $r^2$-supplement
    of~$r$.  

    In particular, if $r^2\leq M$, then $r$ is the first primitive
    lattice vector in~$S$ after $r_n$ 
    with norm${}\leq M$.  By inductive hypothesis the same statement
    holds with $r_1$ in place of~$r_n$.  This justifies step~\ref{StepNotPromisedDone} if it applies.  
    (Before the canonicalization in this step, 
    $l$ is the canonical $r^2$-supplement of $r$, which might not be
    the canonical $M$-supplement.)

    So suppose $r^2>M$.  Then we return to step~\ref{StepNotPromisedNonexistence},
    so $(r_{n+1},l_{n+1})$ is defined.  Along the way we have
    proven that it satisfies the inductive hypotheses.
    All that remains is to show that the algorithm terminates.
    This holds because 
    $r_1^2,r_2^2,\dots$ is strictly decreasing, and $L$ has only
    finitely many possible norms${}\leq r_1^2$ of spacelike vectors.
\end{proof}

Finally we can solve the anisotropic case of Problem~\ref{ProbEdgewalking}.  

\begin{algorithm}[{\sf Anisotropic($M$,$r$,$l$)}]
    \label{AlgorithmAnisotropic}
    Suppose $L$ is anisotropic, $r_0\in S$ is a primitive lattice vector, $M>0$, and
    $l_0$ is the canonical $r_0^2$-supplement of~$r_0$.  If $L$ has no spacelike vectors of
    norm${}\leq M$, then this algorithm says so.  Otherwise, it
    returns a generator  $g$ of
    $\SOup(L)$, and a nonempty sequence $\calR=(r_1,r_2,\dots,r_n)$ of vectors, such that:
    $$
    r_1,\dots,r_n,g(r_1),\dots,g(r_n),g^2(r_1),\dots,g^2(r_n),\dots
    $$
    is the sequence, in ascending order, of all
    primitive lattice vectors in~$S$ that have norm${}\leq M$
    and come after~$r_0$.

    \begin{enumerate}
        \item\relax{\sf [Nonexistence?]}
            \label{StepAnisotropicNonexistence}
            Set $(r_1,l_1)\leftarrow{\sf NotPromised}(M,r_0,l_0)$,
unless  ${\sf NotPromised}(M,r_0,l_0)$ reports nonexistence.  In the latter case,
            stop: $L$ has no spacelike vectors with norm${}\leq M$.  
        \item\relax{\sf [Initialize]} 
            \label{StepAnisotropicIntialize}
            \begin{enumerate} 
                \item
                    set $\calR$ to be the one-term sequence $(r_1)$,
                \item
                    set $(r,l)\leftarrow(r_1,l_1)$, and 
                \item
                    set $A\leftarrow\twobytwo{r^2}{r\cdot l}{r\cdot l}{l^2}$.
            \end{enumerate}
        \item\relax{\sf [Enter main loop]}
            \label{StepAnisotropicEnterLoop}
            Set $(r,l)\leftarrow{\sf Promised}(M,r,l)$ and then
            $l\leftarrow{\sf Canonical}(M,r,l)$.  
        \item\relax{\sf [Done?]}
            \label{StepAnisotropicDone}
            If the inner product matrix of~$(r,l)$ equals~$A$, then set
            $g$ to be the linear transformation sending $(r_1,l_1)$
            to $(r,l)$, and return $g$ and $\calR$.
        \item\relax{\sf [Append \& repeat]}
            \label{StepAnisotropicAppendRepeat}
            Append $r$ to~$\calR$, and go to step~\ref{StepAnisotropicEnterLoop}.
    \end{enumerate}
\end{algorithm}

\begin{proof}
    If {\sf NotPromised} reports nonexistence in step~\ref{StepAnisotropicNonexistence}, then $L$ has no
    spacelike vectors of norm${}\leq M$, as claimed.  
    Otherwise, step~\ref{StepAnisotropicIntialize} sets $r$ to be
    the first primitive lattice vector of~$S$ with norm${}\leq M$ after $r_0$,
    and sets $l$ to be its canonical $M$-supplement. 

    Now consider the $n$th entry to step~\ref{StepAnisotropicEnterLoop},
    at which point the initial subsequence $r_1,\dots,r_n$ of~$\calR$
    has been defined, $r$ is another name for $r_n$, $l$ and $l_1$
    are  the canonical $M$-supplements of~$r$ and $r_1$,
    and the following inductive hypotheses hold.
    First, each $r_{i\geq1}$ is the first primitive lattice vector in~$S$ of norm${}\leq M$
    that comes after $r_{i-1}$.  
    Second, no element of $\SOup(L)$ sends $r_1$ to any of $r_2,\dots,r_n$.
    {\sf Promised($M$,$r$,$l$)} makes sense in step~\ref{StepAnisotropicEnterLoop} because $r_n^2\leq M$.

    Upon entering step~\ref{StepAnisotropicDone},
    $r$ is the first primitive lattice vector in~$S$ of norm${}\leq M$ that
    comes after $r_n$, and $l$ is its canonical $M$-supplement.  If
    step~\ref{StepAnisotropicDone} applies, then $g$ is an isometry.  It lies in $\SO(L)$ because 
    $(r_1,l_1)$ and $(r,l)$ are oriented bases.  Together with $r_1,r\in S$
    this shows that $g\in\SOup(L)$.  It is a generator for $\SOup(L)$
    because otherwise a generator would send $r_1$ to some primitive lattice
    vector in~$S$ that lies strictly between $r_1$ and $r$.  This would contradict
    the  fact that $r_1,\dots,r_n$ are
    all the norm${}\leq M$ primitive lattice vectors that come strictly after $r_0$ and
    strictly before $r$.  Since $g(r_1)=r$, this also justifies the rest of our claims.

    If step~\ref{StepAnisotropicDone} does not apply, then step~\ref{StepAnisotropicAppendRepeat} defines
    $r_{n+1}=r$, and $l$ is already its canonical  $M$-supplement.  No element
    of $\SOup(L)$ sends $r_1$ to $r_{n+1}$, because it would also send
    $l_1$ to $l$   and the algorithm would have stopped at the previous step.  So the
    induction continues after returning to step~\ref{StepAnisotropicEnterLoop}.  The algorithm
    terminates, because otherwise the infinite sequence $r_1,r_2,\dots$ would
    contain $g(r_1)$, where $g$ is a generator of $\SOup(L)$.  But then the
    algorithm would have stopped via step~\ref{StepAnisotropicDone}.
\end{proof}

\section{Edgewalking}
\label{SecEdgewalking}

\noindent
This section contains the heart of the paper.  
Informally stated, we solve the following problem: given a corner of
a Weyl chamber, and an edge emanating from it, find the corner at the
other end of that edge.  If the chamber has finite volume, then repeating
this process will find all the corners and simple roots.  

Before making a formal statement, we say what we mean by a corner.
Suppose $L$ is
a Lorentzian lattice of dimension $n+1$, $\Pi$ is its root
system, $C$ is a chamber for $W(L)$,
and  $\{\alpha_i\}$ are simple roots for~$C$.
An ordinary corner of~$C$ means a point $k\in C\cap H^n$ for
which $k^\perp\cap\Pi$ has rank~$n$.  An ideal corner means
a point $k\in C\cap\partial H^n$, such that every
horosphere centered there meets $C$ in a compact set.  An edge
means the obvious thing; we only mention that we take edges
to contain their ideal endpoints.  Topologically each edge is
a closed segment, and if it is incident at a corner, then the
corner is an endpoint of the segment.

\begin{lemma}[Edges end at corners]
    \label{LemEndsOfEdges}
    In the notation above, suppose $C$ is a chamber for $W(L)$,
    $k$ is a corner of~$C$, and $e$ is an edge incident to~$k$.
    Then the other end of~$e$ is also a corner of~$C$.
\end{lemma}

\begin{proof}
    Write $u_1,\dots,u_{n-1}$ for
    the simple roots of~$C$ orthogonal to~$e$, and
    $k'$ for the other end of~$e$.  If $k'$ is an ordinary point, then
    the fact that $e$ ends there forces the existence of a
    simple root~$\beta$ which is orthogonal to $k'$ but not~$e$.
    So $k'^\perp\cap\Pi$ contains $u_1,\dots,u_{n-1},\beta$, whose
    span has rank~$n$.

    Now suppose $k'$ is ideal.  
   We use the upper half-space model, with $k'$ at vertical
    infinity.  Each horosphere $H$ 
    centered at~$k'$ appears as a horizontal plane, and the 
    mirrors $u_i^\perp$ appear as vertical halfplanes whose
    intersection is a vertical half-line.
    The $O(L)$-stabilizer of $k'$ contains a normal subgroup~$\Z^n$
    consisting of horizontal translations, 
    that acts cocompactly on $H$.   The $\Z^n$-images of the $u_i^\perp$
    are roots of~$L$.  Their orthogonal complements cut $H$ into bounded
    pieces.  $C\cap H$ lies in the closure of one of  these pieces, 
    and is closed.  So it is  compact.
\end{proof}

\begin{problem}[Edgewalking]
    \label{ProbEdgewalking}
    Suppose $L$ is a Lorentzian lattice of rank~$n+1$, 
    $\Pi$ is its root system,
    $W=W(L)$, $C$ is a Weyl chamber of~$W$, $k$ is a corner of $C$,
    and $e$ is an edge of~$C$ incident to~$k$.  
    Find the corner~$k'$ at the other end of~$e$, and the
    simple roots of $C$ that are orthogonal to~$k'$.
\end{problem}

\noindent
Our solution uses the following notation:

$u_1,\dots,u_{n-1}$ are the simple roots of~$C$ that are orthogonal to~$e$;

$\D$ is the normed Dynkin diagram of $u_1,\dots,u_{n-1}$;

$U$ is the real span of $u_1,\dots,u_{n-1}$;

$P$ is the plane in $L\tensor\R$ orthogonal to~$U$;

$\pi_U,\pi_P$ are the orthogonal projection maps to~$U$ and $P$;

$\frac12P$ is the open half-plane in~$P$ corresponding to~$e$;

$S$ is the sector of spacelike and lightlike vectors in~$\frac12P$;

$r_0$ is the primitive lattice vector in~$S$ that is orthogonal to~$k$;

$L_N$, for any given $N>0$, is the constraint lattice $L\cap\frac{N}{2}L^*$.

\noindent
In section~\ref{SecShortVectors} we defined a partial order on~$S$.  We refine it here
to a total order, by declaring that if two members of $S$ are proportional,
then the one closer to~$0$ precedes the one further from~$0$. 

\begin{SolutionToEdgewalking}
    \leavevmode\nopagebreak

    Step 1: Find a list $\calN$ of positive numbers,  that contains the
    norms of all the roots of~$L$. (For example, appeal to Lemma~\ref{LemRootNorms}.)

    \smallskip
    Step 2: For each $N\in\calN$, find the set $\calD_N$ of all norm~$N$
    spherical extensions $\D\to\D'$ of normed Dynkin diagrams.  (See Section~\ref{SecNormedDynkinDiagrams}.)

    \smallskip
    Step 3: For each $N\in\calN$, and each extension $\D\to\D'$ in $\calD_N$,
    define the following (if
    the stated conditions are met).
    \begin{enumerate}
        \item
            \label{DefinitionsEdgewalkingu}
 $u_{N,\D'}$ is the element of $U$ having inner products
            with $u_1,\dots,u_{n-1}$ as specified by the extension $\D\to\D'$.  
    \item
        \label{DefinitionsEdgewalkingResidualNorm}
     The ``residual norm'' $R_{N,\D'}$ is $N-u_{N,\D'}^2$. 
\item
    \label{DefinitionsEdgewalkingResidualCoset}
    If 
            some $v\in L_N$ satisfies $\pi_U(v)=u_{N,\D'}$, then the ``residual coset''
            $C_{N,\D'}$ is $\pi_P(v)+(L_N\cap P)$. (This is independent of~$v$.)
\item
    \label{DefinitionsEdgewalkingRootResidue}
    If $C_{N,\D'}$ is defined, 
            and some vector of~$S\cap C_{N,\D'}$ has norm $R_{N,\D'}$ and 
            comes after $\R_{>0}r_0\sset S$,
            then the ``root residue'' 
            $r_{N,\D'}$ is the first such vector.
        \item
            \label{DefinitionsEdgewalkingCandidate}
            If $r_{N,\D}$ is defined, then the ``candidate''
            $\alpha_{N,\D'}$ is $u_{N,\D'}+r_{N,\D'}$.
    \end{enumerate}

    \smallskip
    Step 4: If no candidates exist, then
    stop:
    the other endpoint $k'$ of~$e$
    is the ideal point at the end of the ray emanating from~$k$ along~$e$.
    Our solution to the cuspidal batch~$0$ problem (problem~\ref{ProbCuspidalBatch0}) 
    extends $\{u_1,\dots,u_{n-1}\}$ 
    to a set of simple roots for $\Pi\cap k'^\perp$.

    \smallskip
    Step 5: Define $\alpha$ as the (unique) candidate $\alpha_{N,\D'}$ satisfying:
    \begin{enumerate}
        \item
            \label{ItemEdgewalkingDistToVertex}
    among all candidates, it minimizes $-k\cdot r_{N,\D'}/\sqrt{R_{N,\D'}}$,
    \item
        \label{ItemEdgewalkingPriority}
        among all candidates satisfying~\eqref{ItemEdgewalkingDistToVertex}, it minimizes
        the priority $-k\cdot \alpha_{N,\D'}/\sqrt{N}
            =-k\cdot r_{N,\D'}/\sqrt{N}$.
    \item
        \label{ItemEdgewalkingSmallestNorm}
        among all candidates satisfying~\eqref{ItemEdgewalkingDistToVertex} and \eqref{ItemEdgewalkingPriority}, it
            has smallest norm.
    \end{enumerate}
    Then $\{u_1,\dots,u_{n-1},\alpha\}$ is the set of
    simple roots for $\Pi\cap k'^\perp$ corresponding to the chamber~$C$.
    In particular,
    the other endpoint
    $k'$ corresponds to $\alpha^\perp\cap P\sset\R^{n,1}$.
\end{SolutionToEdgewalking}

    The ``residual'' objects in step~3
    refer to
    properties of the projection to~$P$ of a (hypothetical) root $\alpha$ with
    norm~$N$ and diagram extension~$\D'$.  The point is that it is the
    part of $\alpha$ not already described by $u_{N,\D'}$.  

\begin{remark}[Simpler implementation, better performance]
    As written, we compute each candidate separately.  
    This makes it easier to explain the proof, but
    gives up a great deal because
    the algorithms in section~\ref{SecShortVectors} are perfectly suited to
    searching for all candidates with fixed~$N$ simultaneously.
    Furthermore, most candidates have no chance of satisfying
    \eqref{ItemEdgewalkingDistToVertex} in step~\ref{ItemEdgewalkingPriority}, and their computation 
    can be avoided.   

    Suppose $N\in\calN$ is fixed, and write $M$ for the largest
    residual norm among all $\D'\in\calD_N$.  
    Apply the algorithms in section~\ref{SecShortVectors} to the lattice
    $\pi_P(L_N)$.  That is, if $P\cap L_N$ is isotropic then use
    {\sf Promised} repeatedly; otherwise use {\sf Anisotropic} once.
    This gives the list $\calV$ of all primitive 
     vectors $r\in S\cap\pi_P(L_N)$ that have norm${}\leq M$  and
     come after~$r_0$, in increasing
    order.  For each such~$r$, consider all of its positive integer
    multiples $s$ that have norm${}\leq M$.  For each $s$,
    check whether there exists $\D'\in\calD_N$
    with $s^2=R_{N,\D'}$ and $s\in C_{N,\D'}$.  If so,
    then we have found $r_{N,\D'}=s$, hence the candidate
    $\alpha_{N,\D'}$.  If any candidates arise from $r$ in
    this way, then we halt the
    search for norm~$N$ candidates after finding them.
    The point is that if $r'$ comes after $r$ in $\calV$,
    then $r'^\perp$ is further from $k$ than $r^\perp$ is.
    Therefore any candidates arising from $r'$ would be discarded
    in step~5\eqref{ItemEdgewalkingDistToVertex}.

    Furthermore, once a candidate has been found, it can be
    used to abandon the searches for candidates of other norms.
    These searches should be abandoned when
    any additional candidates they might find would automatically be
    discarded in \eqref{ItemEdgewalkingDistToVertex} or \eqref{ItemEdgewalkingPriority} of step~5.
\end{remark}

\begin{remark}
    Using {\sf Anisotropic}
    in the previous remark returns a list
    of vectors $r_1,\dots,r_m\in S\cap L_N$ and a generator~$g$
    for $\SOup(\pi_P(L_N))$, such that
    $$
    \calV=\bigl(r_1,\dots,r_m,g(r_1),\dots,g(r_m),g^2(r_1),\dots,g^2(r_m),\dots\bigr)
    $$
    One cannot examine just  $r_1,\dots,r_m$ when checking
    whether they could be 
    root residues $r_{N,\D'}$, because for example $g(r_1)$
    may lie in a different coset of $P\cap L_N$ in $\pi_P(L_N)$
    than $r_1$ does.  Instead, one
    must examine all terms of $\calV$ through
    $g^{m-1}(r_m)$, where $g^m$ is a
    power of~$g$ that is the restriction of an
    isometry of $L_N$.  Sufficient conditions for this are that
    $g^m$ preserves $P\cap L_N$
    and acts trivially on $\pi_P(L_N)/(P\cap L_N)$.  Therefore
    $m$ can be found from~$g$. 
    If  one of
    $r_1,\dots,r_m,\dots,g^{m-1}(r_1),\dots,g^{m-1}(r_N)$
    has a positive integer multiple $s$ that
    lies in $C_{N,\D'}$ and has norm $R_{N,\D'}$,
    then $r_{N,\D'}$ is the first such multiple (under
    our total ordering on~$S$).
    Otherwise
    $r_{N,\D'}$ does not exist.
\end{remark}

\begin{proof}
    First we claim that if a candidate
    $\alpha_{N,\D'}$ is defined,
    then it
    is an almost-root of~$L$.  This follows from Lemma~\ref{LemConstraintLattices},
    because $\alpha_{N,\D'}$ has norm $R_{N,\D'}+u_{N,\D'}^2=N$
    and lies in $C_{N,\D'}+u_{N,\D'}\sset L_N$.  

    Furthermore, $\alpha_{N,\D'}^\perp$ meets the
    hyperbolic line containing~$e$, because $\alpha_{N,\D'}$
    extends $u_1,\dots,u_{n-1}$ to the spherical Dynkin diagram~$\D'$.  
    Together with the fact that $r_{N,\D'}$ comes after $\R_{>0}r_0$ under
    the ordering on~$S$, this shows that $\alpha_{N,\D'}^\perp$
    meets the ray $E$ emanating from $k$ along~$e$, and has negative inner product
    with~$k$.
    We can now finish the proof in the case that $k'$ is ideal.  
    In this case, no mirror of $W(L)$ can meet~$E$ in
    hyperbolic space.
    Therefore no $\alpha_{N,\D'}$ can exist, and
    the algorithm recognizes in step~4 
    that $k'$ is 
    the ideal endpoint of~$E$.

    So suppose $k'$ is timelike, and write $\beta$ for the root which
    extends $u_1,\dots,u_{n-1}$ to the system of simple roots for $\Pi\cap k'^\perp$
    whose Weyl chamber contains~$C$.  
    We claim $\beta$ is a candidate.
    Define $r=\pi_P(\beta)$ and $N=\beta^2$, and 
    let $\D'$
    be the normed Dynkin diagram of $u_1,\dots,u_{n-1},\beta$.  
    That $r$ comes after $\R_{>0} r_0\sset S$ follows from the fact that
    $\beta^\perp$ meets $E$ but not its initial endpoint~$k$.
    Obviously $k\cdot r=k\cdot\beta$, which is negative by the
    definition of~$\beta$.
    Because $\beta\in L_N$ and
    $\pi_U(\beta)=u_{N,\D'}$, the residual coset $C_{N,\D'}$ is defined
    and contains $r$.
    Because $\beta$ can serve a ``some vector'' in step~3\eqref{DefinitionsEdgewalkingRootResidue}, 
    $r_{N,\D'}$ and $\alpha_{N,\D'}$ are defined.  Now, $r_{N,\D'}$ cannot
    precede $r$ in the ordering on~$S$, because then the mirror
    $\alpha_{N,\D'}^\perp$ would
    cut the hyperbolic segment $\overline{kk'}=e$.  
    On the other hand,
    the definition of $r_{N,\D'}$ shows that it cannot come
    after $r$.  
    Because $r$ and $r_{N,\D'}$ have the same norm, and neither
    precedes the other, they are equal.
    Therefore $\beta$ is the
    candidate $\alpha_{N,\D'}$.

    Now suppose $\gamma$ is any candidate, and $w=\pi_P(\gamma)$.  
    The hyperbolic distance from $k$ to $E\cap\gamma^\perp$ is at
    least as large as the hyperbolic distance from $k$ to $k'$, or
    else $k'$ would not be the other endpoint of $e$.  Therefore
    $$
    -k\cdot w/\sqrt{w^2} \geq -k\cdot r/\sqrt{r^2}
    $$
    with equality just if $\gamma\perp k'$.  So the candidates satisfying
    \eqref{ItemEdgewalkingDistToVertex} in step~5 
    are exactly the candidates whose mirrors
    pass through~$k'$.  Obviously $\beta$ is among them.

    We saw that these candidates are almost-roots.
    So we may apply Lemma~\ref{LemVariation} 
    to them, taking 
    $W=W(\Pi\cap k'^\perp)$ and using $k$ as the control vector.
    Arguing as in the last step in the proof of
    our solution to Problem~\ref{ProbSphericalBatchZeroIterative} shows that $\beta$ is the unique candidate satisfying
    \eqref{ItemEdgewalkingDistToVertex}--\eqref{ItemEdgewalkingSmallestNorm} of step~5.
\end{proof}

We organize repeated use of the algorithm as follows.  We maintain 
lists of corners found, simple roots found, and rays remaining to
explore.  Initially, these consist of the given corner~$k$, the 
given simple roots for $\Pi\cap k^\perp$, and the rays 
emanating from~$k$ along edges.  (These edges
can be read from the Dynkin diagram of
$\Pi\cap k^\perp$.)  Now we repeat the following process:
walk along the first unexplored
ray $E$,  using our solution to Problem~\ref{ProbEdgewalking}, and then remove $E$ from
the list of rays to explore.  
If the corner $k'$
at the other end of~$E$ is not already known, 
then do the following.  First, add $k'$ to the list of known
corners.  Second, append to the list of known roots
all the simple roots orthogonal to $k'$ that
are not already known.  Third, to the list of unexplored
rays append all the rays emanating from~$k'$, along edges.

\begin{theorem}
    \label{ThmComputesChamber}
    The chamber
    $C$ has finite volume if and only if this process terminates 
    (meaning that after
    some finite number of steps the list of unexplored rays
    becomes empty). 
    In this case, it finds all
    the
    corners and simple roots of~$C$.
\end{theorem}

\begin{proof}
    If walking along an edge leads to a corner already known,
    then the list of unexplored rays shrinks, because one is
    removed and none are added.   
    Therefore, if $C$ has 
    finite volume, hence finitely many corners, then the
    algorithm terminates.
    
    Now suppose the process terminates.
    To prove that all simple roots and corners are
    found, we will use the projective
    model of hyperbolic space.  Suppose $\alpha_1,\dots,\alpha_k$
    are the simple roots found, 
    and let $B=\set{v\in L\tensor\R}{v\cdot\alpha_i\forall i}$
    be the cone they define.  We regard $PB$ as a 
    convex polytope  in the affine subspace of $P(L\tensor\R)$,
    got by discarding a suitable hyperplane of $P(L\tensor\R)$.
    (For example, discard the  orthogonal complement of
    a vector in the interior of~$B$.)  Note that $PB$ contains~$C$.
    We claim that 
    all vertices of~$PB$ are corners of~$C$.  Certainly
    one vertex is, namely the initial corner (call it~$v_1$).

    Now suppose  $v_2,\dots,v_l$ are given,  each $v_i$ being a
    vertex of~$B$ adjacent to~$v_{i-1}$. We claim that $v_l$ is a corner
    of~$C$, not just a vertex of~$B$.  
    By induction, $v_{l-1}$ is a corner of~$C$.  
    (The base case~$v_1$ is already known.)
    Whenever
    we find a new vertex, we also find all simple roots of~$C$
    orthogonal to it.  It follows that  every edge of~$PB$
    emanating from~$v_{l-1}$ contains an edge of $C$ emanating
    from~$v_{l-1}$.
    Because the algorithm has terminated, we have already 
    walked along the ray from $v_{l-1}$  toward~$v_l$, 
    found the second corner $c$ of~$C$
    on it, and found all simple roots of~$C$
    orthogonal to~$c$.  In particular, $c$ is also
    a vertex of~$B$, which forces $c=v_l$, finishing the proof
    of the inductive step.

    Because $PB$ is a finite-sided polytope, every vertex of~$PB$
    lies at the end of such a sequence $v_1,\dots,v_l$.
    Therefore all vertices of 
    $PB$
    lie in $\overline{H^n}$.  So their convex hull $PB$ does too, and
    it follows that $B$ contains no spacelike vectors.  No unfound simple
    root of~$C$ can exist, because it would be a spacelike element of~$B$.
    So $C=PB$.  Since $C$ is the convex hull of finitely many points of
    $\overline{H^n}$, it has finite hyperbolic volume.
\end{proof}

\begin{remark}[Termination in the infinite-volume case]
    \label{RkInfiniteVolumeTermination}
    If $C$ has infinite volume, then 
    our algorithm can fail to find all vertices and simple roots
    of~$C$,
    even if left to run forever.  This
    happens whenever $\partial C$ is disconnected; for examples
    of this see \cite[III]{NikulinRank3}.    In this case, the algorithm finds exactly
    the (infinitely many)
    vertices in the component $K$ of $\partial C$ containing~$k$,
    and the simple roots orthogonal to each of them.  

    When using Vinberg's algorithm,
    the usual strategy for detecting that $C$ has infinite volume,
     effective in principle and often 
     in practice, is to keep searching for simple
    roots until enough are found to detect some
    elements of~$\Oup(L)$
    that preserve~$C$.  If enough are found to
    generate an infinite group, then $C$ must
    have infinite volume.   For example, whether two vertices
    are $\Oup(L)$-equivalent
    can be determined by trying to identify the simple roots at one
    with with simple roots at the other.  

    This strategy works just as well, and probably better, in our setting,
 because finding vertices is a basic part of our approach.
    If $C$ has infinite volume, then the component of $\partial C$
    containing~$k$ has infinitely many vertices, so the $\Oup(L)$-stabilizer
    of that component of $\partial C$ is infinite.  As one finds vertices,
    one can test them for equivalence with known vertices.  Once enough
    vertices are found, infinitely many automorphisms of~$C$ 
    will be visible, so $C$ must have
    infinite volume.
\end{remark}

\end{document}